\documentclass[a4paper,10pt]{article}
\usepackage[latin1]{inputenc}
\usepackage{amsmath}
\usepackage{amsfonts}
\usepackage{amssymb}
\usepackage{dsfont}
\usepackage{mathrsfs}
\usepackage{graphicx}
\usepackage{setspace}
\usepackage{fancyhdr}
\usepackage{amsthm}
\usepackage{empheq}
\usepackage{cases}
\usepackage[all]{xy}

\newtheorem{theorem}{Theorem}[section]
\newtheorem{definition}{Definition}[section]
\newtheorem{lemma}[theorem]{Lemma}
\newtheorem{proposition}[theorem]{Proposition}

\newtheorem{remark}[theorem]{Remark}

\newcommand{\erre}{\mathds{R}}

\newcommand{\partder}[1]{\frac{\partial}{\partial {#1}}}

\newcommand{\gradh}[1]{\nabla_{H^m}{#1}}
\newcommand{\absh}[1]{{\left|#1\right|_{H^m}}}

\newcommand{\ra}{\rightarrow}

\newcommand{\set}[1]{{\left\{#1\right\}}}               
\newcommand{\pa}[1]{{\left(#1\right)}}                  
\newcommand{\sq}[1]{{\left[#1\right]}}                  
\newcommand{\abs}[1]{{\left|#1\right|}}                 
\newcommand{\pair}[1]{\left\langle#1\right\rangle}      

\newcommand{\eps}{\varepsilon}                           

\renewcommand{\hat}[1]{\widehat{#1}}
\renewcommand{\tilde}[1]{\widetilde{#1}}

\title{Articolo}
\author{Marco Magliaro, Luciano Mari, Paolo Mastrolia, Marco Rigoli}

\begin{document}
\title{\textbf{Keller-Osserman type conditions for \\differential
inequalities with gradient terms on the  Heisenberg group}}
\date{}
\maketitle
\scriptsize \begin{center} Dipartimento di Matematica,
Universit\`a
degli studi di Milano,\\
Via Saldini 50, I-20133 Milano (Italy)\\
E-mail addresses: magliaro@mat.unimi.it, lucio.mari@libero.it,\\
paolo.mastrolia@unimi.it, Marco.Rigoli@mat.unimi.it
\end{center}
\begin{abstract}
 \noindent The aim of this paper is to study the qualitative
behaviour of non-negative entire solutions of certain differential
inequalities involving gradient terms on the Heisenberg group. We
focus our investigation on the two classes of inequalities of the
form $\Delta^\varphi u \ge f(u)l(|\nabla u|)$ and $\Delta^\varphi
u \ge f(u) - h(u) g(|\nabla u|)$, where $f,l,h,g$ are non-negative
continuous functions  satisfying certain monotonicity properties.
The operator  $\Delta^\varphi$, called the
\emph{$\varphi$-Laplacian}, can be viewed as a
 natural generalization of the $p$-Laplace operator recently
 considered by various authors in this setting. We prove some
 Liouville theorems introducing two new Keller-Osserman type
 conditions, both extending the classical one which appeared long
 ago in the study of the prototype differential inequality $\Delta
 u \ge f(u)$ in $\erre^m$. Furthermore, we show sharpness of our
 conditions when we specialize to the case of the $p$-Laplacian.
 Needless to say, our results continue to hold, with the obvious
 minor modifications, also in the Euclidean space.
\end{abstract}
 \normalsize
\begin{section}{Introduction and main results}
To state our main results  we first need to recall some
preliminary facts  and to introduce the notations that we shall
use in the sequel.\par
Let $H^m$ be the Heisenberg group of dimension $2m+1$, that is,
the Lie group with underlying manifold  $\erre^{2m+1}$and group
structure defined as follows: for all $q, q' \in H^m$, $q=(z,
t)=(x_1, \ldots, x_m, y_1, \ldots, y_m, t)$,  $q' = (z', t') =
(x_1', \ldots, x_m', y_1', \ldots, y_m', t')$,
\[
q \circ q' = \pa{z+z', t+t' + 2\sum_{i=1}^m \pa{y_ix_i' -
x_iy_i'}}.
\]
A basis for the Lie algebra of left-invariant vector fields on
$H^m$ is given by
\begin{equation}
  X_j = \partder{x_j} + 2y_j \partder{t}, \qquad Y_j = \partder{y_j}
  -2x_j\partder{t}, \qquad \partder{t}
\end{equation}
for $j=1, \ldots, m$. This basis satisfies Heisenberg's canonical
commutation relations for position and momentum,
\begin{equation}\label{commutation}
  \sq{X_j, Y_k} = -4 \delta_{jk}\partder{t},
\end{equation}
all other commutators being zero. It follows that the vector
fields $X_j, Y_k$ satisfy Hörmander's condition, and the
\emph{Kohn-Spencer Laplacian}, defined as
\begin{equation}
  \Delta_{H^m} = \sum_{j=1}^m \pa{X_j^2 + Y_j^2}
\end{equation}
is hypoelliptic by Hörmander's theorem (see \cite{HOR}).

\noindent In $H^m$ there are a ``natural'' origin $o=(0, 0)$ and a
distinguished \emph{distance function from zero} defined, for
$q=(z, t) \in H^m$, by
\begin{equation}
  r(q)= r(z, t) = \pa{|z|^4 + t^2}^{1/4}
\end{equation}
(where $|\cdot|$ denotes the Euclidean norm in $\erre^{2m}$),
which is homogeneous of degree $1$ with respect to the Heisenberg
dilations $(z, t) \mapsto (\delta z, \delta^2 t),\ \delta>0$. This
gives rise to a distance on $H^m$, called the \emph{Koranyi
distance}, and defined by
\begin{equation}
  d(q, q') = r(q^{-1} \circ q') , \qquad q,q' \in H^m.
\end{equation}
We set
\[
B_R(q_o) = \set{q \in H^m : d(q,q_o)<R}
\]
to denote the (open) \emph{Koranyi ball} of radius $R$ centered at
$q_o$. We simply use $B_R$ for balls centered at $q_o=o$. The
\emph{density function} with respect to $o$ is the function
\begin{equation}
  \psi(q) = \psi(z, t) = \frac{|z|^2}{r^2(z, t)}\quad \text{for}\,\,
  q=(z, t) \neq o;
\end{equation}
note that $0 \leq \psi \leq 1$. For $u \in C^1(H^m)$, the
\emph{Heisenberg gradient}  $\nabla_{H^m}u$ is given by
\begin{equation}
  \nabla_{H^m}u = \sum_{j=1}^m (X_ju)X_j + (Y_ju)Y_j,
\end{equation}
(so that, for $f\in C^1(\erre)$, $\nabla_{H^m}f(u) =
f'(u)\nabla_{H^m}u$), and a \emph{\ $\cdot$ product} on the span
of $X_j, Y_j$  is defined, for $W = w^j X_j + \widetilde{w}^j
Y_j$, $Z = z^j X_j + \widetilde{z}^j Y_j$ by the formula
\begin{equation}
   W \cdot Z = \sum_{j=1}^m
  w^jz^j+ \widetilde{w}^j\widetilde{z}^j.
\end{equation}
By definition, $\abs{\gradh{u}}_{H^m}^2 = \gradh{u} \cdot
\gradh{u}$, and we have the validity of the Cauchy-Schwarz
inequality
\begin{equation}
  \absh{\gradh{u} \cdot \gradh{v}} \leq
  \absh{\gradh{u}}\absh{\gradh{v}}.
\end{equation}

The distance function $r$ satisfies the following fundamental
relations involving $\psi$:
\begin{equation}
  \Delta_{H^m}r = \frac{2m+1}{r}\psi \qquad \text{in}\,\, H^m \backslash
  \{o\},
\end{equation}
\begin{equation}
  |\nabla_{H^m}r|_{H^m}^2  = \psi \quad \quad \quad \ \text{in}\,\,H^m \backslash \{o\}.
\end{equation}
Recently, some authors (see, for example, \cite{DOM}, \cite{CDG}
and \cite{BPT}) have studied a generalization of the Kohn
Laplacian, defined, for $p \in [2, +\infty)$, by
\begin{equation}
  \Delta^p_{H^m}u = \sum_{j=1}^m\sq{
  X_j\pa{|\nabla_{H^m}u|_{H^m}^{p-2}X_ju} +
  Y_j\pa{|\nabla_{H^m}u|_{H^m}^{p-2}Y_ju}}, \qquad u \in C^2(H^m),
\end{equation}
which can be considered as a natural $p$-Laplace operator in the
setting of the Heisenberg group.\par
In this paper we consider a further generalization, which we shall
call \emph{$\varphi$-Laplacian}, $\Delta^\varphi_{H^m}$, defined
for $u \in C^2(H^m)$ as follows:
\begin{equation}
  \Delta^\varphi_{H^m}u = \sum_{j=1}^m\sq{X_j\pa{|\nabla_{H^m}u|_{H^m}^{-1}\varphi\pa{|\nabla_{H^m}u|_{H^m}}X_ju}
  +
  Y_j\pa{|\nabla_{H^m}u|_{H^m}^{-1}\varphi\pa{|\nabla_{H^m}u|_{H^m}}Y_ju}},
\end{equation}
where $\varphi$ satisfies the structural conditions
\begin{equation}\label{Phi}\tag{$\Phi$}
  \begin{cases}
    \varphi \in C^0(\erre_0^+) \cap C^1(\erre^+), \quad  \varphi(0)=0, \\ \quad \varphi' >0 \qquad \text{ on } \, \erre_0^+.
  \end{cases}
\end{equation}
This family of operators, containing the $p$-Laplacian (obtained
with the choice $\varphi(t)=t^{p-1}$, $p>1$), has been recently
studied in the context of Riemannian geometry (see, for example,
\cite{PRS_MEM} for motivations and further references). Although
we shall focus our attention on this generalization, the main
example we  keep in mind is the $p$-Laplacian itself, to which an
entire section is devoted.\\
\par
The aim of this paper is to study weak (in the sense of Subsection
\ref{sotsezdue} below) non-negative entire solutions of
differential inequalities of the form
\begin{equation}\label{filaplaciano}
  \Delta^\varphi_{H^m}u \geq f(u)l(\absh{\gradh{u}}),
\end{equation}
where $f$ and $l$ satisfy respectively the following conditions:
\begin{equation}\label{F}\tag{$F$}
  \begin{cases}
    f \in C^0(\erre^+_0), \qquad f>0 \text{ on } \, \erre^+; \\
    f \text{ is  increasing on } \erre^+_0;
  \end{cases}
\end{equation}
\begin{equation}\label{L}\tag{$L$}
  \begin{cases}
  l \in C^0(\erre^+_0), \quad l>0 \, \text{ on }\, \erre^+; \\
  l \text{ is } C\text{-monotone non-decreasing on } \erre_0^+; \\
   \end{cases}
\end{equation}

  We recall that $l$ is said to be \emph{$C$-monotone non decreasing} on
  $\erre_0^+$ if, for some $C \geq 1$,
  \[
  \sup_{s \in [0, \,t]}l(s) \leq C l(t), \qquad \forall \, t \in \erre_0^+.
  \]
  Clearly, if $l$ is monotone non decreasing on $\erre_0^+$, then
  it is $1$-monotone non-decreasing on the same set; in fact the above
  condition allows a controlled oscillatory behaviour of $l$ on
  $\erre_0^+$.
  To express our next requests, from now on we assume that
 \begin{equation}\label{non_integr_inf}\tag{$\Phi$ \& $L$}
    \frac{t\varphi'(t)}{l(t)}  \in L^1(0^+)\backslash L^1(+\infty), \quad
    \frac{\varphi(t)}{l(t)} = o(1) \quad \text{as } \ t \rightarrow 0^+.
  \end{equation}
  Note that often (e.g. in the case of the $p$-Laplacian) the latter condition
  directly assures integrability
  at $0^+$ in the former. We define
  \begin{equation}\label{def_K}
    K(t) = \int_0^t \frac{s\varphi'(s)}{l(s)}\,\mathrm{d}s;
  \end{equation}
  \noindent observe that $K : \erre^+_0 \ra \erre_0^+$ is a
  $C^1$-diffeomorphism with
  \[
  K'(t) = \frac{t\varphi'(t)}{l(t)}>0,
  \]
  thus the existence of the increasing inverse $K^{-1} : \erre_0^+ \ra \erre^+_0$. Finally we set
  \[
  F(t) = \int_0^t f(s)\,\mathrm{d}s.
  \]

\begin{definition}
 The \textbf{\emph{generalized Keller-Osserman condition}} for
inequality
$$
\Delta^\varphi_{H^m} u \ge f(u) l(\absh{\gradh{u}})
$$
is the request:
  \begin{equation}\label{generalized_K-O}\tag{$KO$}
    \frac{1}{K^{-1}\pa{F(t)}} \in L^1(+\infty).
  \end{equation}
\end{definition}
\vspace{0.2cm}

Note that, in the case of the $p$-Laplace operator and $l \equiv
  1$, \eqref{generalized_K-O} coincides with the well known
  Keller-Osserman condition for the $p$-Laplacian, that is,
  $\frac{1}{F(t)^{1/p}} \in L^1(+\infty)$.\\
In order to deal with the presence of the density function $\psi$
in the version of our inequalities that we shall describe below,
we need to assume two ``relaxed homogeneity'' requests on
$\varphi'$ and $l$:
\begin{equation}\label{cond_ult_phi}\tag{$\Phi2$}
        s\varphi'(st)\le Ds^\tau\varphi'(t), \qquad \forall \ s \in [0, 1], \  t \in \erre^+_0,
\end{equation}
\begin{equation}\label{L2}\tag{$L2$}
   s^{1+\tau}l(t) \leq \Lambda \,l(st), \qquad \ \  \, \forall \ s \in [0, 1], \  t \in \erre^+_0,
\end{equation}
for some positive constants $D,\Lambda >0$ and $\tau \ge 0$. We
stress that \eqref{L2} is a mild requirement: for example, it is
satisfied by every $l(t)$ of the form
$$
l(t) = \sum_{k=0}^N C_k t^{\nu_k}, \quad N \in \mathds{N}, \qquad C_k \ge 0, \quad -\infty < \nu_k \le 1+\tau \quad \text{for every } k.
$$
Indeed, since $s\le 1$ we have
$$
l(st) = \sum_{k=0}^N C_k s^{\nu_k}t^{\nu_k} \ge \sum_{k=0}^N C_k s^{1+\tau}t^{\nu_k} = s^{1+\tau} l(t).
$$
Note also that, if \eqref{L2} is true for some $\tau_o$, then it
also holds for every $\tau \ge \tau_o$. This is interesting in the
case of the $p$-Laplacian, which trivially satisfies
\eqref{cond_ult_phi} for every $0\le \tau \le p-1$. In this case
the choice $\tau = p-1$ is the least demanding on $l(t)$. We also
observe that the coupling of \eqref{cond_ult_phi} and \eqref{L2}
does not automatically imply the integrability at $0^+$ in
\eqref{non_integr_inf}. For instance if $\varphi(t)=t^\tau$ and
$l(t)=t^{\tau+1}$, then \eqref{cond_ult_phi} and \eqref{L2} are
satisfied, but $\frac{t\varphi'(t)}{l(t)}\not\in L^1(0^+).$
\vspace{0.4cm}\\
We shall prove the following Liouville-type
result:
  \begin{theorem}\label{non_existence_th}
    Let $\varphi, \ f, \ l$ satisfy \eqref{Phi}, \eqref{F},
    \eqref{L} and
    \eqref{non_integr_inf}. Suppose also the validity of the relaxed homogeneity conditions
    \eqref{cond_ult_phi}, \eqref{L2}. If the generalized Keller-Osserman condition \eqref{generalized_K-O} holds, then
    every solution
    $0 \le u \in C^1(H^m)$ of
    \begin{equation}\label{disug_delta}
      \Delta^\varphi_{H^m}u \geq f(u)l(\absh{\gradh{u}}) \qquad
      \text{on}\,\,\,
      H^m
    \end{equation}
    is constant. Moreover, if \, $l(0)>0$, then $u \equiv 0$.
  \end{theorem}
   The proof is achieved through the construction of a ``radial''
supersolution $v$ of \eqref{disug_delta}
  \noindent(see the next section for the precise definition) on an annular
  region $B_T \backslash B_{t_0}, \,\, 0 < t_0 < T$, which is small near $\partial
  B_{t_0}$ and blows up at $\partial B_T$. A careful comparison between $u$ and $v$ allows us to conclude that
  $u$ must necessarily be constant. As opposed to
  Osserman's approach (see \cite{OSS}), in order to construct the
  supersolution we have not tried to solve the
  radialization of \eqref{disug_delta}, since the presence of the
  gradient term may cause different behaviours near the first
  singular time. Roughly speaking, even if we could prove the local existence
  of a radial solution in a neighborhood of zero (which is not
  immediate due to the singularity of $1/r$ and possibly of
  $\varphi'$ in $0$), we cannot be sure that, in case the interval
  of definition is $[0,T)$, $T<+\infty$, the solution blows up at
  time $T$: \emph{a priori}, it may even happen that the solution remains bounded,
  but the first derivative blows up, giving rise to some sort of cusp.
  The necessity of excluding this case led us to a
  different approach: a blowing-up supersolution is explicitly
  constructed, exploiting directly
  the Keller-Osserman condition. Beside being elementary, this alternative method also reveals the reason why
  ($KO$) is indeed natural as an optimal condition for the existence or
  non-existence of solutions.\\
  \par
  As it will become apparent from the proof of Theorem
  \ref{non_existence_th} below, the result can be restated on the
  Euclidean space $\mathds{R}^m$ getting rid of request
  \eqref{cond_ult_phi} and (L2), which are related to the density function $\psi$. Indeed we have

  \begin{theorem}\label{non_existence_th_euclid}
    Let  $\varphi, \ f, \ l$ satisfy \eqref{Phi}, \eqref{F}, \eqref{L},
    \eqref{non_integr_inf} and the generalized
    Keller-Osserman condition \eqref{generalized_K-O}. Let $u \in C^1(\erre^m)$ be
    a non-negative solution of
    \begin{equation}
      \Delta^\varphi_{\erre^m} u = \operatorname{div}\pa{\abs{\nabla u}^{-1}
      \varphi\pa{\abs{\nabla u}}\nabla
      u}\geq f(u)l(\abs{\nabla u}) \qquad \text{on}\,\,\erre^m.
    \end{equation}
    Then $u$ is constant. Moreover, if $l(0) > 0$, then $u \equiv
    0$.
  \end{theorem}

  To show the sharpness of \eqref{generalized_K-O}, we produce a global unbounded subsolution of \eqref{filaplaciano}
  when \eqref{generalized_K-O} is violated. For simplicity we  only deal with the case of the $p$-Laplacian
   and we  prove the following:

\begin{theorem}\label{th_existence}
Assume the validity of \eqref{F} and \eqref{L}. Suppose that
\begin{equation}\label{PeL}\tag{$p$ \& $L$}
  \frac{t^{p-1}}{l(t)} = o(1) \quad \text{as }\ t \ra 0^+ \quad ,
  \quad l(t) \le B_1 + B_2 t^\mu \quad \forall \ t \in \erre_0^+,
\end{equation}
where $B_1, B_2>0 $ and $0 \le \mu < 1$. Assume also the relaxed
homogeneity condition
\begin{equation}\label{L2p}\tag{$L2_p$}
  l(t)s^p \leq \Lambda \,l(st) \qquad \forall \ s \in [0, 1], \  t \in \erre^+_0.
\end{equation}
 Then the following conditions are equivalent:

 \begin{itemize}
    \item [i)] there exists a non-negative, non-constant
        solution $u \in C^1(H^m)$ of inequality
        $\Delta^p_{H^m} u \ge f(u) l(\absh{\gradh{u}});$
    \item [ii)] $\displaystyle \frac{1}{K^{-1}(F(t))} \not\in
        L^1(+\infty)$.
 \end{itemize}
\end{theorem}

\vspace{0.4cm} As for Theorem \ref{non_existence_th}, we can state
the analogous result in Euclidean setting: in this latter case,
assumption \eqref{L2p} is unnecessary.
We would like to stress that the subsolution constructed to prove
the necessity part of the Keller-Osserman condition is unbounded.
This fact is not accidental: indeed, in Section
\ref{sect_nonexistofbounded} we shall prove that, under all the
assumptions of Theorem \ref{non_existence_th} but
\eqref{generalized_K-O}, bounded subsolutions still have to be
constant.

 In the last part of the paper we show how the techniques
introduced  can be implemented
  to study differential inequalities of the form
  \begin{equation}\label{equaz_col_meno}
    \Delta^\varphi_{H^m}u \geq f(u) - h(u)g\pa{\absh{\gradh{u}}},
  \end{equation}
  where the functions appearing in the RHS of the above are
  non-negative. The main results obtained are Theorem
  \ref{TH_non_exist_meno}, that is, triviality of the solutions
  in the general setting under an appropriate Keller-Osserman condition, and Theorem \ref{esistenzameno} for the
  $p$-Laplace operator, where we show the sharpness of the condition in analogy with Theorem \ref{th_existence}.
  Details appear in Section \ref{more} below.

\end{section}

\begin{section}{Preliminaries}\label{sect_foreplay}
The aim of this section is to introduce an explicit formula for
the $\varphi$-Laplacian acting on radial functions and the
appropriate notion of weak solution of differential inequalities
of the type of \eqref{filaplaciano} or, more generally,
\eqref{equaz_col_meno}.

\begin{subsection}{``Radialization'' of the
$\varphi$-Laplacian}\label{sect_radial}

Consider a \emph{radial} function, that is, a function of the form

\begin{equation}
  u(q) = \alpha(r(q)), \qquad q \in H^m,
\end{equation}
 where $\alpha : \erre_0^+ \ra
\erre, \, \alpha \in C^2(\erre_0^+$).

Now, a straightforward but somewhat lengthy computation yields the
expression:
\begin{equation}\label{radializz_deltaphi}
  \Delta^\varphi_{H^m}u = \sqrt{\psi}\sq{\sqrt{\psi}\varphi'\pa{|\alpha'(r)|\sqrt{\psi}}\alpha''(r) +
  \frac{2m+1}{r} \operatorname{sgn}\alpha'(r) \,\varphi\pa{|\alpha'(r)|\sqrt{\psi}}}.
\end{equation}
It is worth to stress the following property, which allows us to
shift the origin for the Koranyi distance from $o$ to any other
point $q_0$: if we denote with $\bar{r}(q)= d(q_0, q) =
r(q_0^{-1}\circ q)$, a calculation shows that
\[
\sq{X_j(\bar{r})}(q)= \sq{X_j(r)}(q_0^{-1} \circ q), \qquad
\sq{Y_j(\bar{r})}(q)= \sq{Y_j(r)}(q_0^{-1} \circ q),
\]
hence we obtain the invariance with respect to the left
multiplication
\begin{equation}\label{philapl_rsegnato}
\Delta^\varphi_{H^m}\pa{\alpha \circ \bar{r}}(q)=
\Delta^\varphi_{H^m}(\alpha \circ r)(q_0^{-1} \circ q).
\end{equation}
The above relation will come in handy in what follows.
\end{subsection}

\begin{subsection}{Weak formulation}\label{sotsezdue}

In this section we derive a weak formulation for the differential
inequality \eqref{filaplaciano}. In order to simplify the
notation, let us first introduce the function
\begin{equation}\label{At}
 A(t) = t^{-1}\varphi(t), \qquad A(t) \in
C^0(\erre^+).
\end{equation}
With the help of the matrix $B=B(q)$ (see \cite{BRS}, pg. 294),
defined by

\begin{displaymath}\label{matrice}
B(q)=B(z,t)= \left(
\begin{array}{c|c}
I_{2m}&\begin{array}{c} 2y_1\\\vdots\\2y_m\\-2x_1\\ \vdots\\-2x_m
\end{array}\\  \hline
\begin{array}{cccccc}
2y_1 & \cdots & 2y_m & -2x_1 & \cdots & -2x_m
\end{array}&
4\abs{z}^2
\end{array}
\right),
\end{displaymath}
we can write the $\varphi$-Laplacian in divergence form. Indeed,
indicating from now on with $\operatorname{div}$, $\nabla$  and
$\pair{\phantom{a},\phantom{b}}$ respectively the ordinary
Euclidean divergence, gradient and scalar product in
$\erre^{2m+1}$, given $u \in C^2(H^m)$ we have
\begin{align*}
  \Delta^\varphi_{H^m}u  =& \sum_j\sq{
  X_j\pa{A(\absh{\gradh{u}})X_j u} + Y_j\pa{A(\absh{\gradh{u}})Y_j
  u}}=\\
  =&\sum_j[A(\absh{\gradh{u}})X_j(X_ju) +
  X_j(A(\absh{\gradh{u}}))X_ju + \\
  &\phantom{\sum[}+ A(\absh{\gradh{u}})Y_j(Y_ju) +
  Y_j(A(\absh{\gradh{u}}))Y_ju ]= \\ =&
  A(\absh{\gradh{u}})\operatorname{div}(B \nabla u) +
  \gradh{A(\absh{\gradh{u}})} \cdot \gradh{u},
\end{align*}
where with $B\nabla v$ we mean the vector in $\erre^{2m+1}$ whose
components in the standard basis $\partder{x_j},
\partder{y_j}, \partder{t}$ are given by the matrix multiplication
of $B$ with the components of $\nabla v$ in the same basis. Having
made this precise, it is easy to see that $B \nabla v =
\gradh{v}$.
 Now, a standard check shows that, for
$u, v \in C^1(H^m)$,
\begin{equation}
  \pair{\nabla u, B \nabla v} = \gradh{u} \cdot \gradh{v}.
\end{equation}
Then, going back to the previous computation we have
\begin{align*}
  \Delta^\varphi_{H^m}u  &= A(\absh{\gradh{u}})\operatorname{div}(B \nabla u) +
  \pair{\nabla A(\absh{\gradh{u}}),B\nabla u}=\\
  &= \operatorname{div}\pa{A(\absh{\gradh{u}})B\nabla u},
\end{align*}
which is the desired expression. Note that, when $\varphi(t)=t$,
the above becomes the well-known formula (see, e.g., \cite{GALA}
and \cite{BRS}) for the Kohn-Spencer Laplacian, that is,
$\Delta^\varphi_{H^m}u = \operatorname{div}(B\nabla u).$ It
follows that \eqref{filaplaciano} can be interpreted in the weak
sense as follows: for all $\zeta \in C^\infty_0(H^m)$, $\zeta \geq
0$, we have
\begin{align*}
  \int_{\erre^{2m+1}}\zeta \Delta^\varphi_{H^m}u &= \int_{\erre^{2m+1}}\zeta
  \operatorname{div}(A(\absh{\gradh{u}})B\nabla u)= \\ &= -
  \int_{\erre^{2m+1}}A(\absh{\gradh{u}})\pair{B\nabla u,
  \nabla\zeta}= \\ &= - \int_{\erre^{2m+1}}A(\absh{\gradh{u}})
  \gradh{u} \cdot \gradh{\zeta},
\end{align*}
and thus the weak form is
\begin{equation}\label{weak_form}
  -\int_{\erre^{2m+1}}A(\absh{\gradh{u}}) \gradh{u} \cdot
  \gradh{\zeta} \geq \int_{\erre^{2m+1}}f(u)
  l(\absh{\gradh{u}})\zeta
\end{equation}
as expected. Hence, an \emph{entire weak classical solution of
\eqref{filaplaciano}} is a function $u \in C^1(H^m)$ such that,
for all $\zeta \in C_0^{\infty}(H^m), \zeta \geq 0$,
\eqref{weak_form} is satisfied. A similar definition of course
holds for the differential inequality \eqref{equaz_col_meno}.
\end{subsection}
\end{section}

\begin{section}{Proof of Theorem \ref{non_existence_th}}\label{sect_mainresult}
 In order to prove Theorem \ref{non_existence_th} we shall need a comparison theorem
and a maximum principle which are well-known for the Kohn-Spencer
Laplacian (see \cite{BON}). Here we briefly prove the
corresponding statements for the $\varphi$-Laplacian that we shall
use below, basing on ideas taken from \cite{PRS} and \cite{PUSE}.
Throughout Subsections \ref{subsect_comp} and \ref{subsect_max} we
shall assume \eqref{Phi} and \eqref{cond_ult_phi}.

\begin{subsection}{Comparison principle}\label{subsect_comp}
  \begin{proposition}\label{comparison}
    Let $\Omega \subset \subset H^m$  be a
    relatively compact domain with $C^1$ boundary. Let $u, v \in
    C^0(\overline{\Omega}) \cap C^1(\Omega)$ satisfy

    \begin{equation}\label{comp}
      \begin{cases}
       \Delta^\varphi_{H^m}u \geq \Delta^\varphi_{H^m}v  &\text{on } \Omega\\
       u \leq v &\text{on } \partial \Omega.
      \end{cases}
    \end{equation}
    Then $u \leq v$ on $\Omega$.

    \begin{proof}
    The proof basically follows the one in \cite{PRS_MEM}
    pp.85--86. However, we reproduce the steps for the sake of
    completeness.  Let $w=v-u.$ By contradiction assume that there exists $\bar{q}\in\Omega$
    such that $w(\bar{q})<0$, and let $\eps>0$ be such that
    $w(\bar{q})+\eps<0$. The function $w_\eps=\min\{w+\eps,0\}$ has
    compact support in $\Omega$, hence $-w_\eps\geq0$ is an
    admissible  Lipschitz test function. The weak definition of \eqref{comp}, together with the
    divergence form of $\Delta^\varphi_{H^m}$, reads:
    \begin{align}\label{dish1}
      0&\geq\int_\Omega\pair{\abs{\gradh{v}}_{H^m}^{-1}\varphi(\absh{\gradh{v}})B\nabla v-
      \abs{\gradh{u}}_{H^m}^{-1}\varphi(\absh{\gradh{u}})B\nabla u, \nabla
      w_\eps}=\nonumber\\
      =&\int_E\pair{\abs{\gradh{v}}_{H^m}^{-1}\varphi(\absh{\gradh{v}})B\nabla v-
      \abs{\gradh{u}}_{H^m}^{-1}\varphi(\absh{\gradh{u}})B\nabla u,
      \nabla(v-u)},
    \end{align}
    where $E=\set{q:w(q)<-\eps}$.
    We denote with $h$ the integrand in \eqref{dish1}. With the aid of the
    Cauchy-Schwarz inequality we have
    \begin{equation}\label{integranda_pos}
      h\geq\sq{\varphi(\absh{\gradh{v}})-\varphi(\absh{\gradh{u}})}(\absh{\gradh{v}}-\absh{\gradh{u}})\geq0,
    \end{equation}
    where the latter inequality is due to the monotonicity of
    $\varphi$.\\
    It follows from \eqref{dish1} and \eqref{integranda_pos} that $0\geq\int_\Omega
    h\geq0$, hence $h=0$ a.e. on $\Omega.$\\
    This implies that $\absh{\gradh{u}}=\absh{\gradh{v}}$ on
    $E$, and therefore
    \begin{align*}
    0=& h=\abs{\gradh{u}}_{H^m}^{-1}\varphi(\absh{\gradh{u}})\pair{B\nabla(v-u),\nabla(v-u)}=\\
    =&\abs{\gradh{u}}_{H^m}^{-1}\varphi(\absh{\gradh{u}})\abs{\nabla(v-u)}_{H^m}^2.
    \end{align*}
    This shows that
    \begin{equation}\label{weps}
      \abs{\nabla(w_\eps)}_{H^m}^2=0,
    \end{equation}
    whence $w_\eps$ is constant. Indeed, from \eqref{weps} we have
    $X_j(w_\eps)=Y_j(w_\eps)=0$ for every $j=1\ldots m$, and using
    the commutation law \eqref{commutation} we also have $\partial w_\eps/\partial
    t=0;$ recalling the definition of $X_j$ and $Y_j$, all the components
    of the Euclidean gradient of $w_\eps$ vanish, proving the
    constancy of $w_\eps$. Since
    $w_\eps(\bar{q})<0={w_\eps}_{\mid\partial\Omega}$ we reach the
    desired contradiction.
    \end{proof}
  \end{proposition}

\end{subsection}

\vspace{0.4cm}
\begin{subsection}{Maximum principle}\label{subsect_max}
  \begin{proposition}\label{maximum_principle}
    Let $\Omega \subset H^m$  be a
     domain. Let $u \in C^0(\overline{\Omega}) \cap C^1(\Omega)$ satisfy
    \begin{equation}\label{subfi}
      \Delta^\varphi_{H^m}u \geq 0 \quad \text{in}\,\,
      \Omega
    \end{equation}
    and let $\displaystyle u^*=\sup_\Omega u$.
    If $u(q_M)=u^*$ for some $q_M \in \Omega$, then $u \equiv u^*.$
    \end{proposition}

\begin{proof}

By contradiction, suppose the existence of a solution $u$ of
\eqref{subfi} and of $q_M\in\Omega$ such that $u(q_M)=u^*$, but
$u\not\equiv u^*$. Set $\Gamma = \set{q\in \Omega: u(q) = u^*}$.
Let $\delta>0$ and define
\begin{equation}
   \Omega^+=\set{q\in\Omega:u^*-\delta<u(q)<u^*}; \qquad
   \Gamma_\delta = \set{q\in \Omega: u(q) = u^*-\delta};
\end{equation}
note that $\partial\Omega^+\cap\Omega= \Gamma \cup \Gamma_\delta$.
Let $q'\in\Omega^+$ be such that
\begin{equation}
  d\pa{q',\Gamma}<d(q',\Gamma_\delta), \qquad d(q',\Gamma)<d(q',\partial \Omega)
\end{equation}
(this is possible up to choosing $q'$ sufficiently close to
$q_M$). Let $B_R(q')$ be the largest Koranyi ball centered at $q'$
and contained in $\Omega^+$. Then, by construction $u<u^*$ in
$B_R(q')$ while $u(q_0)=u^*$ for some $q_0\in\partial B_R(q').$
Since $q_0$ is an absolute maximum for $u$ in $\Omega$, we have
$\nabla u(q_0)=0.$

Now we construct an auxiliary function. Towards this aim, we
consider the annular region
\begin{equation}
  E_R(q')=\overline{B_R(q')}\setminus B_{R/2}(q')\subset\Omega^+;
\end{equation}
we fix  $a\in(u^*-\delta,u^*)$ to be determined later and consider
the following problem
\begin{equation}
  \left\{\begin{array}{lll}\label{sist}
  \sq{\varphi\pa{z'}}'+\frac{2m+1}{t}\varphi\pa{z'}\leq0\quad\text{in }(R/2,R)\\
  z(R/2)=a,\quad z(R)=u^*\\
  u^*-\delta<z\leq u^*,\quad z'>0\quad\text{in }[R/2,R].
\end{array}\right.
\end{equation}
Notice that, for example, the function
\begin{equation}
  z(t)=\int_{R/2}^t\varphi^{-1}\pa{\frac{c}{s^{2m+1}}}\,\mathrm{d}s+a
\end{equation}
satisfies \eqref{sist} for some suitable constant $c$.\\
Using the invariance property \eqref{philapl_rsegnato}, such a
function gives rise to a $C^2$-solution $v(q)=z(\bar{r}(q))$,
where $\bar{r}(q)= r(q'^{-1} \circ q)$, of
\begin{equation}\label{sistv}
  \left\{\begin{array}{lll}
  \Delta^\varphi_{H^m}v\leq0\quad\text{in }E_R(q')\\
  v=a \quad\text{on }\partial B_{R/2}(q'),\quad v=u^*\quad \text{on }\partial B_{R}(q') \\
  u^*-\delta<v\leq u^*.
\end{array}\right.
\end{equation}
Indeed by hypothesis \eqref{cond_ult_phi} we have
\begin{equation}
  \Delta^\varphi_{H^m}v\leq
  D\pa{\sqrt{\psi}}^{\tau+1}\set{\sq{\varphi\pa{z'}}'+\frac{2m+1}{t}\varphi\pa{z'}}\leq0.
\end{equation}
It is important to point out that there exists a positive constant
$\lambda >0$ such that
\begin{equation}
  \langle\nabla v,\nabla \bar{r}\rangle=z'(\bar{r})\abs{\nabla \bar{r}}^2 \geq \lambda >0 \quad \text{on }\,\, \partial E_R(q');
\end{equation}
this follows since $\bar{r}$ differs from $r$ by a translation of
the Heisenberg group (that is, a diffeomorphism), and $\abs{\nabla
r}^2 = \frac{1}{r^6}\pa{|z|^6 + \frac{t^2}{4}}$ only vanishes at
the origin $o$. Next we choose $a\in\pa{u^*-\delta,u^*}$ close
enough to $u^*$ so that
  $u\leq v$ on $\partial B_{R/2}\pa{q'}$: this is possible since $\partial B_{R/2}\pa{q'}\subset
  \subset \Omega^+$ and thus $\max_{\partial
B_{R/2}\pa{q'}}u<u^*$. Now $u,v\in C^0(\overline{E_R(q')}) \cap
C^1(E_R(q'))$ and, since $v\equiv u^*$ on $\partial B_{R}\pa{q'},$
they satisfy
\begin{equation}
  \begin{cases}
     \Delta^\varphi_{H^m}u \geq \Delta^\varphi_{H^m}v  &\text{on }E_R\pa{q'}\\
     u \leq v &\text{on }\partial E_R\pa{q'}.
   \end{cases}
 \end{equation}
Then by Proposition \ref{comparison} we have $u\leq v$ on $E_R\pa{q'}.$\\
Let us consider the function $v-u$: it satisfies $v-u\geq0$ on
$E_R(q')$ and $v(q_0)-u(q_0)=u^*-u^*=0$, so that
$\langle\nabla(v-u),\nabla \bar{r}\rangle(q_0)\leq 0.$ Therefore
\begin{equation}
  0=\langle\nabla u,\nabla \bar{r}\rangle(q_0)\geq\langle\nabla v,\nabla \bar{r}\rangle(q_0)>0,
\end{equation}
a contradiction.
\end{proof}

\begin{remark}
  \emph{Obviously, one can state an analogous minimum principle using
  the substitution $v(q) = -u(q)$; however, a direct proof of the
  minimum principle following the above steps reveals some further
  difficulties due to the density function, which is not bounded
  from below away from zero.}
\end{remark}

\end{subsection}

\begin{subsection}{Construction of the supersolution}
In order to construct the radial supersolution for
\eqref{filaplaciano} we  point out the validity of the next
technical Lemma. We refer to the Introduction for notations and
properties.
  \begin{lemma}\label{sigma_lemma}
    Let $\sigma \in (0, 1]$; then the generalized
    Keller-Osserman condition \eqref{generalized_K-O} implies
    \begin{equation}
      \frac{1}{K^{-1}(\sigma F(t))} \in L^1(+\infty).
    \end{equation}
    \begin{proof}
      We perform the  change of variables $t= s\sigma$ to have
      \[
      \int^{+\infty}\frac{\mathrm{d}s}{K^{-1}(\sigma F(s))}= \sigma^{-1}\int^{+\infty}\frac{\mathrm{d}t}{K^{-1}(\sigma
      F(\sigma^{-1}t))}.
      \]
      Since $f$ and $K^{-1}$ are increasing by assumption, we
      get
      \[
      F(\sigma^{-1}t)=\int_0^{\sigma^{-1}t} f(z)\,\mathrm{d}z = \sigma^{-1}\int_0^t
      f(\sigma^{-1}\xi)\,\mathrm{d}\xi \geq \sigma^{-1}\int_0^t
      f(\xi)\,\mathrm{d}\xi = \sigma^{-1}F(t)
      \]
      and
      \[
      K^{-1}(\sigma F(\sigma^{-1}t)) \geq K^{-1}(F(t)),
      \]
      thus
      \begin{equation}\label{sigmalemma2}
      \int^{+\infty}\frac{\mathrm{d}s}{K^{-1}(\sigma F(s))}\leq
      \sigma^{-1}\int^{+\infty} \frac{\mathrm{d}t}{K^{-1}(F(t))}<+\infty.
      \end{equation}
    \end{proof}
  \end{lemma}
  Here is the construction of the supersolution.

  \begin{lemma}\label{supersol}
    Suppose the validity of \eqref{Phi}, \eqref{F}, \eqref{L},
    \eqref{non_integr_inf} and of the Keller-Osserman
    \eqref{generalized_K-O}.
    Fix $0<\eps<\eta$ and $0<t_0<t_1$. Then, for every $\widetilde{B}>0$ there exist $T>t_1$ and a strictly
    increasing, convex function
    \[\alpha : [t_0, T) \ra [\eps,
    +\infty)
    \]
    satisfying
    \begin{equation}\label{prop_supersol}
      \begin{cases}
        \pa{\varphi(\alpha')}' + \frac{2m+1}{t}\varphi(\alpha')
        \leq \widetilde{B} f(\alpha)l(\alpha'); \\
        \alpha(t_0) = \eps, \quad\alpha(t_1) \leq \eta; \\
        \alpha(t) \uparrow +\infty \text{ as } t \ra T^-.
      \end{cases}
    \end{equation}
    \begin{proof}
    Consider $\sigma\in(0,1]$ to be determined later and choose $T_\sigma>t_0$ such that
    \begin{equation*}
    T_\sigma-t_0=\int^{+\infty}_\eps\frac{\mathrm{d}s}{K^{-1}(\sigma
    F(s))}.
    \end{equation*}
    Note that the RHS is well defined by Lemma \ref{sigma_lemma} and, since it diverges as $\sigma\ra0^+,$ up to
    choosing $\sigma$ sufficiently small we can shift $T_\sigma$
    in such a way that $T_\sigma>t_1.$ We implicitly define the
    $C^2$-function $\alpha(t)$ by requiring
    \begin{equation*}
    T_\sigma-t=\int^{+\infty}_{\alpha(t)}\frac{\mathrm{d}s}{K^{-1}(\sigma
    F(s))} \qquad \text{on }\ [t_0, T_\sigma).
    \end{equation*}
    We observe that, by construction, $\alpha(t_0)=\eps$ and,
    since $K>0,$ $\alpha(t)\uparrow+\infty$ as $t\ra T_\sigma.$ A
    first differentiation yields
    \begin{equation*}
    \frac{\alpha'}{K^{-1}(\sigma F(\alpha))}=1,
    \end{equation*}
    hence $\alpha$ is monotone increasing and $\sigma
    F(\alpha)=K(\alpha').$ Differentiating once more we deduce
    \begin{equation*}
    \sigma
    f(\alpha)\alpha'=K'(\alpha')\alpha''=\frac{\alpha'\varphi'(\alpha')}{l(\alpha')}\alpha''.
    \end{equation*}
    Cancelling $\alpha'$ throughout we obtain
    \begin{equation*}
    \sq{\varphi\pa{\alpha'}}'=\varphi'\pa{\alpha'}\alpha''=\sigma
    f(\alpha)l(\alpha');
    \end{equation*}
    thus, integrating on $[t_0,t]$,
    \begin{equation*}
    \varphi\pa{\alpha'(t)}=\varphi\pa{\alpha'(t_0)}+\sigma\int_{t_0}^tf(\alpha(s))l(\alpha'(s))\,\mathrm{d}s.
    \end{equation*}
    Using \eqref{F} and \eqref{L} we deduce the following
    chain of inequalities:
    \begin{align*}
    &\sq{\varphi(\alpha')}' + \frac{2m+1}{t}\varphi(\alpha')=\\
    &=\sigma f(\alpha)l(\alpha')+\frac{2m+1}{t}\varphi(\alpha'(t_0))+
    \frac{2m+1}{t}\sigma\int_{t_0}^tf(\alpha(s))l(\alpha'(s))\,\mathrm{d}s=\\
    &=\sq{\sigma+\frac{2m+1}{t}\frac{\varphi(\alpha'(t_0))}{f(\alpha(t))l(\alpha'(t))}+
    \frac{2m+1}{t}\frac{\sigma\int_{t_0}^tf(\alpha(s))l(\alpha'(s))\,\mathrm{d}s}
    {f(\alpha(t))l(\alpha'(t))}}f(\alpha(t))l(\alpha'(t))\le\\
    &\le\sq{\sigma+\frac{2m+1}{t}\frac{\varphi(\alpha'(t_0))}{f(\alpha(t_0))l(\alpha'(t_0))}+
    \frac{2m+1}{t}\frac{\sigma f(\alpha(t))l(\alpha'(t))(t-t_0)}
    {f(\alpha(t))l(\alpha'(t))}}f(\alpha(t))l(\alpha'(t)),
    \end{align*}
    that is,
    \begin{equation}\label{dis_super}
    \sq{\varphi(\alpha')}' + \frac{2m+1}{t}\varphi(\alpha')\le
    \sq{\frac{2m+1}{t_0}\frac{\varphi(\alpha'(t_0))}{f(\alpha(t_0))l(\alpha'(t_0))}+
    2(m+1)\sigma}f(\alpha(t))l(\alpha'(t)).
    \end{equation}
    Since $K(0)=0$, $\alpha'(t_0)=K^{-1}(\sigma F(\eps))\ra0$ as
    $\sigma\ra0$, and using \eqref{non_integr_inf}, choosing $\sigma$ small enough we can estimate the
    whole square bracket with $\widetilde{B}$ to show the validity of the first of \eqref{prop_supersol}.\\
    It remains to prove that, possibly with a further reduction of
    $\sigma,$ $\alpha(t_1)\le\eta.$ From the trivial identity
    \begin{equation*}
    \int^{+\infty}_{\alpha(t_1)}\frac{\mathrm{d}s}{K^{-1}(\sigma
    F(s))}=T_\sigma-t_1=(T_\sigma-t_0) + (t_0-t_1)=\int^{+\infty}_{\eps}\frac{\mathrm{d}s}{K^{-1}(\sigma
    F(s))}+(t_0-t_1)
    \end{equation*}
    we deduce
    \begin{equation*}\int^{\alpha(t_1)}_\eps\frac{\mathrm{d}s}{K^{-1}(\sigma
    F(s))}=t_1-t_0.
    \end{equation*}
    It suffices to choose $\sigma$ such that $\int^{\eta}_\eps\frac{\mathrm{d}s}{K^{-1}(\sigma
    F(s))}>t_1-t_0;$ then obviously $\alpha(t_1)\le\eta.$ This
    completes the proof of the Lemma.
    \end{proof}
  \end{lemma}
\end{subsection}

\begin{subsection}{Last step of the proof}\label{principale}

We denote with $u^*=\sup u$ and we first suppose that
$u^*<+\infty$. We reason by contradiction and assume $u\not\equiv
u^*$; by Proposition \ref{maximum_principle} $u < u^*$ on $H^m$.
Choose $r_0 >0$ and define
\[
u_0^* = \sup_{\overline{B}_{r_0}} u < u^*.
\]
Fix $\eta >0$ sufficiently small such that $u^* - u_0^* > 2\eta$,
and choose $\tilde{q} \in H^m \backslash \overline{B}_{r_0}$ such
that $u(\tilde{q}) > u^* - \eta$.\\
We then define $\tilde{r}= r(\tilde{q})$ and  we construct the
radial function $v(q) =\alpha(r(q))$ on $B_T \backslash B_{r_0}$,
with $\alpha$ and $T>\tilde{r}$ as in Lemma \ref{supersol},
$\widetilde{B}= 1/(\Lambda D)$, and satisfying the further
requirement:
\[
\eps \leq v \leq \eta \qquad \text{on } B_{\tilde{r}}
\backslash \overline{B}_{r_0}.
\]
We observe that $v$ is a supersolution for (\ref{filaplaciano}).
Towards this aim, first we note that by integration, \eqref{Phi}
and $s \in [0, 1]$, \eqref{cond_ult_phi} implies the inequality
\begin{equation}\label{conseg_cond_ult_phi}
  \varphi(st) \leq D s^\tau\varphi(t), \quad t \in \erre_0^+, \,\,
  s\in[0, 1].
\end{equation}
Next, considering the radial expression
\eqref{radializz_deltaphi}, using \eqref{L}, \eqref{cond_ult_phi},
\eqref{conseg_cond_ult_phi} and Lemma \ref{supersol} we have
\begin{align*}
\Delta^\varphi_{H^m}\alpha(r(q))&=\sqrt{\psi}\sq{\sqrt{\psi}\varphi'\pa{\alpha'(r)\sqrt{\psi}}\alpha''(r)
+\frac{2m+1}{r}\varphi\pa{\alpha'(r)\sqrt{\psi}}}\le\\
&\le\pa{\sqrt{\psi}}^{1+\tau}D\sq{\varphi'\pa{\alpha'(r)}\alpha''(r)
+\frac{2m+1}{r}\varphi\pa{\alpha'(r)}}\le\\
&\le\pa{\sqrt{\psi}}^{1+\tau}D\sq{\frac{1}{\Lambda D}f(\alpha(r))l(\alpha'(r))}\le\\
&\le
f(\alpha(r))l(\alpha'(r)\sqrt{\psi})=f(\alpha(r))l(\absh{\gradh{\alpha(r)}}).
\end{align*}
 Moreover
\[
u(\tilde{q}) - v(\tilde{q}) > u^* - \eta  - \eta = u^* - 2\eta,
\]
and, on $\partial B_{r_0}$,
\[
u(q) - v(q) \leq u^*_0 - \eps < u^* -  2\eta - \eps.
\]
Thus, considering the difference $u-v$ on the annular region $B_T
\backslash B_{r_0}$, since by construction
\[
u(q) - v(q) \ra -\infty \qquad \text{as } r(q) \ra T,
\]
it follows that $u-v$ attains a positive maximum $\mu$ in
$B_T\backslash \overline{B}_{r_0}$. Let $\Gamma_\mu$ be a
connected component of
\[
\set{q \in B_T\backslash\overline{B}_{r_0} : u(q) - v(q) = \mu}.
\]
Let $\xi \in \Gamma_\mu$ and note that $u(\xi) > v(\xi)$ and
$\absh{\gradh{u (\xi)}}= \absh{\gradh{v (\xi)}}$. As a
consequence, since $f$ is strictly increasing,
\[
\Delta^\varphi_{H^m}u(\xi) \geq f(u(\xi))l(\absh{\gradh{u(\xi)}})>
f(v(\xi))l(\absh{\gradh{ v(\xi)}})\geq \Delta^\varphi_{H^m}v(\xi).
\]
By continuity, there exists an open set $V \supset \Gamma_\mu$
such that
\begin{equation}\label{lu_magg_lv}
\Delta^\varphi_{H^m}u \geq \Delta^\varphi_{H^m}v \qquad \text{on }
V.
\end{equation}
Fix now $\xi \in \Gamma_\mu$ and a parameter $0 < \rho <\mu$; let
$\Omega_{\xi, \rho}$ be the connected component containing $\xi$
of the set
\[
\set{q \in B_{T} \backslash \overline{B}_{r_0} : u(q) > v(q) +
\rho}.
\]
We observe that $\xi \in \Omega_{\xi,\,\rho}$  for every $\rho$
and that $\Omega_{\xi,\, \rho}$ is a nested sequence as $\rho$
converges to $\mu$. We claim that if $\rho$ is close to $\mu$,
then $\overline{\Omega}_{\xi,\, \rho} \subset V$. This can be
shown by a compactness argument such as the following: since
$\Gamma_\mu$ is closed and bounded, there exists $\eps > 0$ such
that $d(V^c, \Gamma_\mu) \geq \eps$. Suppose, by contradiction,
that there exist sequences $\rho_n \uparrow \mu$ and $\{q_n\}$
such that $q_n \in \Omega_{\xi, \rho_n}$ and $d(q_n, \Gamma_\mu) >
\eps$. Then, we can assume that the sequence is contained in
$\Omega_{\xi, \rho_0}$ which, by construction, has compact
closure; passing to a subsequence converging to some
$\overline{q}$, we have by continuity
\begin{equation}\label{distanza_x_segnato}
d(\overline{q}, \Gamma_\mu) \ge \eps,\
\end{equation}
but, on the other hand, $(u-v)(\overline{q}) = \lim_{n} (u-v)(q_n)
\geq \lim_n \rho_n = \mu$, hence $\overline{q} \in \Gamma_\mu$ and
this contradicts \eqref{distanza_x_segnato}. Therefore,
$d(\partial\Omega_{\xi, \rho}, \Gamma_\mu) \ra 0$ as $\rho \ra
\mu$, and the claim is proved.

\noindent On $\partial \Omega_{\xi,\, \rho}$ we have $u(q)=v(q) +
\rho$; since $v(q) + \rho$ solves
\[
\Delta^\varphi_{H^m}(v + \rho) = \Delta^\varphi_{H^m}v \leq
f(v)l(\absh{\gradh{v}}) \leq f(v + \rho)l(\absh{\gradh{(v +
\rho)}}),
\]
by Proposition \ref{comparison},
\[
u(q) \leq v(q) + \rho.
\]
But  $u(\xi) = v(\xi) + \mu$ and $\xi \in \Omega_{\xi,\, \rho}$, a
contradiction. The case $u^* = +\infty$ is easier and can be
treated analogously. This shows that $u \equiv c$, where $c$ is a
non-negative constant; in case $l(0) > 0$ we have $0 =
\Delta^\varphi_{H^m}c \geq f(c)l(0)$. This implies $f(c) = 0$,
hence $c=0$.
\end{subsection}
\end{section}

\begin{section}{Proof of Theorem
\ref{th_existence}}\label{sect_existence}

    This section is devoted to proving the result stated in Theorem
 \ref{th_existence}; first of all we observe that the sufficiency of the Keller-Osserman condition, i.e.
 implication $ii) \Rightarrow i)$, follows from Theorem
 \ref{non_existence_th}. In particular, it is easy to see that
 \eqref{PeL} implies \eqref{non_integr_inf} and that \eqref{L2p}
 implies \eqref{L2}. This latter follows since $\Delta^p_{H^m}$
 satisfies \eqref{cond_ult_phi} for every $0 \le \tau \le p-1$ (as
 we have already pointed out), and $\tau = p-1$ is the
 best choice.
 Our aim is therefore to provide existence of unbounded $C^1$-solutions of inequality
\eqref{disug_delta} under the assumption that
\eqref{generalized_K-O} is not satisfied; this will be achieved
through a careful pasting of two subsolutions defined on
complementary sets. First, we deal with ``\emph{radial stationary
functions}'', that is, functions of the form
\begin{equation*}
  v(q) = w(|z|), \qquad q = (z, t) \in H^m,
\end{equation*}
 where $w : \erre_0^+ \ra
\erre, \, w \in C^2(\erre_0^+)$. Performing computations very
similar to those in Subsection \ref{sect_radial}, we obtain the
following identities:
\[
\absh{\gradh{|z|}} \equiv 1, \qquad \Delta_{H^m}|z| = \frac{2m-1}{|z|},
\]
and thus the expression of the $\varphi$-Laplacian for a radial
stationary function is
\begin{equation}\label{radialstaz}
  \Delta^\varphi_{H^m}v = \varphi'\pa{|w'(|z|)|}w''(|z|) +
  \frac{2m-1}{|z|}\operatorname{sgn}\pa{w'(|z|)}\varphi\pa{|w'(|z|)|}.
\end{equation}
This shows that radial stationary functions in the Heisenberg
group behave  as Euclidean radial ones, and this fact allows us to
avoid dealing with the density
function.\\
Now let $\eps>0$ and $\sigma\ge1$ to be determined later and
define $w_\sigma(t)$ implicitly by
\begin{equation}\label{impli}
    t=\int_\eps^{w_\sigma(t)}\frac{\mathrm{d}s}{K^{-1}(\sigma F(s))},
\end{equation}
The existence of $w_\sigma$  on all $\erre_0^+$ is ensured by the
negation of the Keller-Osserman condition, through the reversing
of Lemma \ref{sigma_lemma}. Observe that $w_\sigma(0)=\eps$ and
$w'_\sigma(t)=K^{-1}(\sigma F(w_\sigma(t)))\ge K^{-1}(F(\eps)) >0$
on $\erre_0^+.$ Define
\begin{equation*}
  t_\sigma=\int_\eps^{2\eps}\frac{\mathrm{d}s}{K^{-1}(\sigma F(s))},
\end{equation*}
so that $w_\sigma(t_\sigma)=2\eps.$\\
The function $u_2(z,t)=w_\sigma(|z|)$ is $C^1$ for $|z|\ge
t_\sigma$ and satisfies
\begin{equation*}
\Delta^\varphi_{H^m}u_2 \geq
  \sigma f(u_2))l(\absh{\gradh{u_2}}) \ge
  f(u_2))l(\absh{\gradh{u_2}})
\end{equation*}
since $\varphi'(w'(|z|))w''(|z|)= \sigma f(w(|z|))l(w'(|z|))$ and
$\varphi$ is non-negative. Unfortunately, $u_2$ is only Lipschitz
on the line $|z|=0$. One might get rid of this problem modifying
the base point of the integral \eqref{impli}, that is,
substituting $\eps$ with $0$, but then one should require
$1/K^{-1}(\sigma F(s)) \in L^1(0^+)$, an assumption which we want
to avoid. Therefore we solve the problem by using a gluing
technique and pasting together a subsolution defined on $|z| \le
t_\sigma$ and a modification of $u_2$ on $|z| \ge t_\sigma$.\\
First of all we consider the Cauchy problem
\begin{equation*}
\left\{ \begin{array}{l}
    [\varphi(\alpha')]' = \Theta
    \quad \text{on } \ [0,+\infty) \\[0.2cm]
    \alpha(0)= \alpha'(0) = 0,
    \end{array}\right.
\end{equation*}
with $\Theta$ a constant to be determined later. This problem has
the solution $\alpha\in C^1(\erre^+_0)\cap C^2(\erre^+)$
\begin{equation*}
  \alpha(t)=\int_0^t\varphi^{-1}(\Theta s) \,\mathrm{d}s;
\end{equation*}
note that $\alpha'(t)>0$ when $t>0$. Choosing
$\Theta=\frac{\varphi(1)}{t_\sigma},$ we have
\begin{equation*}
  \alpha'(t_\sigma)=1 \quad \text{ and }\quad \alpha(t_\sigma)=\int_0^{t_\sigma}\varphi^{-1}(\Theta s) \,\mathrm{d}s\le t_\sigma,
\end{equation*}
and if we fix an $\eps>0$ so that $K^{-1}(F(\eps))>1,$ we also
have that
\begin{equation}\label{adayinthelife}
  \frac{\alpha'(t_\sigma}{w_\sigma'(t_\sigma)}=\frac{1}{K^{-1}(\sigma F(\eps))}\le\frac{1}{K^{-1}(F(\eps))}<1.
\end{equation}
Furthermore, noting that $t_\sigma\ra0$ as $\sigma\ra+\infty,$ up
to choosing $\sigma$ sufficiently large, we have
\begin{equation}\label{helterskelter}
\alpha(t_\sigma)<\eps,
\end{equation}
and since $\Theta=\frac{\varphi(1)}{t_\sigma}\ra+\infty$ as
$\sigma\ra+\infty,$ we can choose $\sigma$ large enough so that
\begin{equation}\label{lucy}
  f(\alpha(t))l(\alpha'(t))\le\Theta\quad\forall t\in[0,t_\sigma].
\end{equation}
This last condition implies that the composition
$u_1(z,t)=\alpha(|z|)$, which is $C^1$ even at $|z|=0,$ satisfies
\begin{equation}\label{u1}
\Delta^\varphi_{H^m} u_1 \ge f(u_1) l(\absh{\gradh{u_1}}) \qquad
\text{on } \ \overline{B}_{T_\sigma}.
\end{equation}
Now we need to glue the solutions $u_1$ and $u_2$ together, and to
this end we define a real $C^2$-function $\gamma_\sigma : [
w_\sigma(t_\sigma), +\infty) \rightarrow [\alpha(t_\sigma),
+\infty)$ such that
\begin{equation}\label{condizioni}
\gamma_\sigma (w_\sigma(t_\sigma)) = \alpha(t_\sigma), \quad 0 <
\gamma_\sigma' \le 1, \quad \gamma_\sigma'(w_\sigma(t_\sigma)) =
\frac{\alpha'(t_\sigma)}{w_\sigma'(t_\sigma)}, \quad \gamma_\sigma'' \ge
0
\end{equation}
Using \eqref{adayinthelife} and $\eqref{helterskelter}$, it is not
hard to see that the above conditions are not contradictory: in
particular from
$\alpha(t_\sigma)<\eps=w_\sigma(0)<w_\sigma(t_\sigma)$ and
$\alpha'(t_\sigma)<w'_\sigma(t_\sigma),$ we see that the requests
involving $\gamma'_\sigma(t)$ are indeed compatible, and it also
holds
\begin{equation}\label{minoret}
\gamma_\sigma(t)\le t \qquad \text{on } \ [w_\sigma(t_\sigma),+\infty).
\end{equation}
Next, we consider the following function, depending on the
parameter $\sigma$:
\begin{equation}
u(z,t)=\left\{ \begin{array}{ll}
u_1(z,t)=\alpha(|z|) & \text{if } \ |z| \in [0,t_\sigma] \\[0.1cm]
(\gamma_\sigma \circ u_2)(z,t)=(\gamma_\sigma \circ w) (|z|)&
\text{if } \ |z| \in [t_\sigma,+\infty)
\end{array}\right.
\end{equation}
Note that, by construction, $u$ has global $C^1$-regularity even
on the cylinder $|z|=t_\sigma$. It remains to prove that, up to
choosing $\sigma$ large enough, it is a subsolution of
\eqref{disug_delta} on the whole $H^m$. By \eqref{lucy}, we only
need to check this for $|z|\ge t_\sigma,$ but unfortunately, in
order to treat this case, we need to assume some homogeneity
conditions which would give $\varphi$ a structure very similar to
the one of the $p$-Laplacian. Therefore, it is more enlightening
to treat directly the $p$-Laplacian case, where things get
simpler. A computation that uses \eqref{condizioni},
\eqref{minoret}, the $C$-monotonicity of $l$ and the monotonicity
of $f$ shows that
\begin{equation}\label{finale}
\begin{array}{l}
\Delta^p_{H^m} u =  \gamma_\sigma'[(\gamma_\sigma')^{p-2}|\gradh{u_2}|_{H^m}^{p-2}]
\Delta_{H^m} u_2 + \\[0.3cm]
 + (\gamma_\sigma')^{p-2}|\gradh{u_2}|_{H^m}^{p-2} \gamma_\sigma''
|\gradh{u_2}|_{H^m}^2 +\\[0.3cm]
 + (p-2)(\gamma_\sigma')^2 (\gamma_\sigma')^{p-3} |\gradh{u_2}|_{H^m}^{p-3} \gradh{u_2}
\cdot \gradh{\absh{\gradh{u_2}}} + \\[0.3cm]
 + (p-2)\gamma_\sigma'\gamma_\sigma'' |\gradh{u_2}|_{H^m}^3 (\gamma_\sigma')^{p-3}
|\gradh{u_2}|_{H^m}^{p-3} \ge \\[0.3cm]
 \ge
(\gamma_\sigma')^{p-1}\pa{|\gradh{u_2}|_{H^m}^{p-2}
\Delta_{H^m} u_2 + (p-2)|\gradh{u_2}|_{H^m}^{p-3} \gradh{u_2}
\cdot \gradh{\absh{\gradh{u_2}}} } = \\[0.3cm]
=  (\gamma_\sigma')^{p-1} \Delta^p_{H^m} u_2 \ge
(\gamma_\sigma')^{p-1} \sigma f(u_2)
l(\absh{\gradh{u_2}}) \ge \\[0.3cm]
 \ge  \displaystyle (\gamma_\sigma'(w(t_\sigma)))^{p-1}
f(\gamma_\sigma \circ u_2)
\frac{\sigma}{C}l(\gamma_\sigma'\absh{\gradh{u_2}})=\\[0.3cm]
 =  \displaystyle \pa{\frac{1}{K^{-1}(\sigma
F(2\eps))}}^{p-1} \frac{\sigma}{C} f(u)
l(\absh{\gradh{u}}).
\end{array}
\end{equation}
The proof is now complete provided we show that
\begin{equation*}
\frac{\sigma}{\left[K^{-1}(\sigma F(2\eps)\right]^{p-1}}
\longrightarrow + \infty \qquad \text{as } \ \sigma \rightarrow +\infty
\end{equation*}
Using the definition of $K$ and the growth condition \eqref{L2p}
we deduce
\begin{equation*}
K(t) = (p-1) \int_0^t \frac{s^{p-1}}{l(s)}\,\mathrm{d}s \ge (p-1) \int_0^t
\frac{s^{p-1}}{B_1+B_2s^\mu}\,\mathrm{d}s \asymp  t^{p-\mu} \quad \text{as }
\ t \rightarrow +\infty.
\end{equation*}
Hence, for some positive constant $\widetilde{C}$ we get
\begin{equation*}
K^{-1}(t) \le \widetilde{C}t^{\frac{1}{p-\mu}}.
\end{equation*}
It follows that, since $\mu<1$,
\begin{equation*}
\frac{\sigma}{\left[K^{-1}(\sigma F(2\eps))\right]^{p-1}} \ge
\frac{\sigma}{\widetilde{C}(\eps)\sigma^{\frac{p-1}{p-\mu}}}
\longrightarrow +\infty \quad \text{as } \ \sigma \rightarrow
+\infty.
\end{equation*}
Up to choosing $\sigma$ sufficiently large we can deduce from
\eqref{finale}
\begin{equation*}
\Delta^p_{H^m} u \ge f(u) l(\absh{\gradh{u}})
\qquad \text{on } |z| \ge t_\sigma
\end{equation*}
and we have the desired conclusion. To end the proof of the
theorem, we note that the $C^1$ regularity of $u$ on the cylinder
$|z|=t_\sigma$ and at the origin $o$
   makes it necessary to
  proceed with the weak formulation. Nevertheless, this is a standard matter because of the continuity of $\gradh{u}$:
  however, for the sake of completeness, let $\xi \in C^\infty_0(H^m)$ and define
  $$
  \begin{array}{l}
  \mathcal{V} = \set{ q=(z,t) \in H^m :  |z|< t_\sigma} \ \cap \ \mathrm{supp}(\xi),\\[0.2cm]
  \mathcal{W}= \set{ q=(z,t) \in H^m :   |z|> t_\sigma} \ \cap \ \mathrm{supp}(\xi), \\[0.2cm]
  \Gamma = \set{ q=(z,t) \in H^m :   |z|= t_\sigma} \ \cap \ \mathrm{supp}(\xi).
  \end{array}
  $$
  Through a suitable partition of
  unity, we can find $\xi_1,\xi_2\in C^\infty_0(H^m)$ such that
  $\xi=\xi_1+\xi_2$ and
  $$
  \mathrm{supp}(\xi_1) \subset \{ (z,t) \in H^m :  |z|< t_\sigma\},
  \quad \mathrm{supp}(\xi_2) \subset \{(z,t) \in H^m :  |z|> \frac{t_\sigma}{2}\}.
  $$
  Because of linearity, it is sufficient to
  show inequality \eqref{disug_delta} for $\xi_1$ and $\xi_2$. For $\xi_1$
  the weak formulation of \eqref{disug_delta} is immediate: indeed, on $\mathrm{supp}(\xi_1)$,
  $u\equiv u_1$ which solves \eqref{u1} weakly. Hence, we only
  need to consider $\xi=\xi_2$. Using the weak
  formulation \eqref{weak_form}, the definition of $u$ on $\mathcal{V},
  \ \mathcal{W}$, and remembering that
  \begin{itemize}
  \item[$(i)$] $u_1,u_2$ are pointwise subsolutions on
      $\mathcal{V}\backslash\{|z|=0\},\ \mathcal{W}$
      respectively, with non-vanishing gradient,
  \item[$(ii)$] $\xi=0$ in a neighborhood of $\{|z|=0\},$
  \end{itemize}
  we deduce, denoting with $\nu_\mathcal{V}$ and
  $\nu_\mathcal{W}$ the (Euclidean) normals to $\partial
  \mathcal{V}$ and $\partial \mathcal{W}$:
  $$
  \begin{array}{l}
  \displaystyle \int_{H^m} |\gradh{u}|_{H^m}^{p-2} \langle B\nabla u,\nabla \xi\rangle =
   \int_{\mathcal{V}}|\gradh{u_1}|_{H^m}^{p-2} \langle B\nabla u_1,\nabla \xi\rangle +
   \\[0.5cm]
   \displaystyle +\int_{\mathcal{W}}(\gamma_\sigma')^{p-1}|\gradh{u_2}|_{H^m}^{p-2} \langle B\nabla u_2,\nabla \xi\rangle
   =
   \displaystyle \int_{\partial \mathcal{V}} |\gradh{u_1}|_{H^m}^{p-2} \langle B\nabla
   u_1,\nu_\mathcal{V} \rangle\xi - \int_\mathcal{V} \xi
   \Delta^p_{H^m} u_1 +\\[0.5cm]
   \displaystyle +
   \int_\mathcal{W} |\gradh{u_2}|_{H^m}^{p-2} \langle B\nabla u_2,\nabla
   ((\gamma_\sigma')^{p-1}\xi)\rangle
   \displaystyle - \int_\mathcal{W}(p-1)(\gamma_\sigma')^{p-2}
   \gamma_\sigma'' |\gradh{u_2}|_{H^m}^p.
   \end{array}
  $$
  Using $\gamma_\sigma''\ge 0$ and the divergence theorem
  for the third addendum, we obtain
$$
  \begin{array}{l}
  \displaystyle \int_{H^m} |\gradh{u}|_{H^m}^{p-2} \langle B\nabla u,\nabla \xi\rangle
  \le
   \displaystyle \int_{\partial \mathcal{V}} |\gradh{u_1}|_{H^m}^{p-2} \langle B\nabla
   u_1,\nu_\mathcal{V} \rangle\xi - \int_\mathcal{V} \xi
   \Delta^p_{H^m} u_1 + \\[0.5cm]
   \displaystyle +
   \int_{\partial \mathcal{W}} |\gradh{(\gamma_\sigma \circ u_2)}|_{H^m}^{p-2} \langle B\nabla
   (\gamma_\sigma \circ u_2),\nu_\mathcal{W} \rangle\xi -
   \int_\mathcal{W} (\gamma_\sigma')^{p-1} \Delta^p_{H^m} u_2 \xi.
  \end{array}
  $$
  Note that the only possibly non-null part of the boundary integrals is
  along $\Gamma$, for which
  $\nu_\mathcal{V}=-\nu_\mathcal{W}$. Since $u$ is $C^1$ on
  $\Gamma$, the boundary terms cancel and, by $(i)$, $(ii)$
  together with the final estimates of \eqref{finale}
  we get
$$
  \begin{array}{l}
  \displaystyle \int_{H^m} |\gradh{u}|_{H^m}^{p-2} \langle B\nabla u,\nabla \xi\rangle
  \le - \int_\mathcal{V} f(u_1)l(\absh{\gradh{u_1}})\xi  + \\[0.5cm]
  \displaystyle - \int_\mathcal{W} \xi f(\gamma_\sigma\circ u_2)
  l(\absh{\gradh{(\gamma_\sigma \circ u_2)}})\xi = - \int_{H^m}
  f(u)l(\absh{\gradh{u}})\xi.
  \end{array}
$$
Therefore $u$ is a weak subsolution, and the proof is complete.

\end{section}

\section{Non-existence of bounded solutions}\label{sect_nonexistofbounded}
The aim of this section is to show that the differential
inequality \eqref{disug_delta} admits no non-constant,
non-negative  bounded solutions in general, that is, even if the
Keller-Osserman condition is not satisfied.

 \begin{theorem}\label{non_existence_th_for_bounded}
    Let $\varphi, f,  l$ satisfy \eqref{Phi},  \eqref{F}, \eqref{L},
    \eqref{non_integr_inf}, \eqref{cond_ult_phi} and \eqref{L2}.
        Then every non-negative bounded $C^1$-solution $u$ of
        \begin{equation}\label{disug_delta_ancora}
      \Delta^\varphi_{H^m}u \geq f(u)l(\absh{\gradh{u}}) \qquad
      \text{on}\,\,\,
      H^m
    \end{equation}
    is constant; moreover, if $l(0)>0$, then $u \equiv 0$.
    \begin{proof}
     Let $u$ be a non-negative bounded solution of \eqref{disug_delta_ancora} and let $u^*= \sup_{H^m} u$.
     We follow the same steps of the proofs of Lemma \ref{supersol} and Theorem
    \ref{non_existence_th} and define a radial supersolution $v(q)=\alpha(r(q))$, where $\alpha : [r_0,
    T_\sigma)\ra \erre^+$ is  defined by
    \begin{equation*}
      T_\sigma - t = \int_{\alpha(t)}^A \frac{\mathrm{d}s}{K^{-1}(\sigma
      F(s))},
    \end{equation*}
    with $A$ any constant greater than $u^*$. Note that, as before, $\alpha(r_0) =\eps, \alpha(\tilde{r}) <
    \eta$ and $\alpha'(t) > 0$ on $[r_0, T_\sigma)$,  while $\alpha(T_\sigma) =
    A$.

    Now choose $r_0$ as in Section  \ref{principale} and consider the difference
    $u-v$ in the annular region $B_{T_\sigma} \backslash
    \overline{B}_{r_0}$; note that, on $\partial B_{r_0}$,
    $u - v < u^* - 2\eta - \eps$, there exists $\tilde{q}$ such
    that $u(\tilde{q}) - v(\tilde{q}) > u^* - 2\eta$, and, on $\partial
    B_{T_\sigma}$, $u-v < u^* - A < 0$. Thus $u-v$ attains a positive maximum $\mu$ at some point
    of $B_{T_\sigma} \backslash \overline{B}_{r_0}$.

    Hereafter, the proof proceeds exactly as that of Theorem
    \ref{non_existence_th}, so we omit the details.
    \end{proof}
    \end{theorem}

\section{More differential inequalities}\label{more}

 The aim of this section is to show that the method used so far allows
 us to treat some other cases; in
 particular, we focus our attention on the
differential inequality \eqref{equaz_col_meno}, that is,

\[
 \Delta^\varphi_{H^m}u \geq f(u) - h(u)g\pa{\absh{\gradh{u}}}.
\]

As a matter of fact, the most interesting case arises when $h\ge
0$ and $g \ge 0$, that is, when we have the action of two opposite
terms and when the standard comparison arguments do not apply.
Indeed, as we shall see, in the generalized Keller-Osserman
condition the terms $h$ and $f$ play very different roles.

\subsection{Basic assumptions and a new adapted Keller-Osserman
condition}

We  collect the following further set of hypotheses:
\begin{equation}\label{H}\tag{$H$}
    h \in C^0(\erre^+), \,\, h(t) \geq 0 \,\,\text{on}\,\,
    \erre^+, \,\, h \in L^1(0^+), \,\, h \text{ monotone non-increasing};
    \end{equation}
\begin{equation}\label{Phi0}\tag{$\Phi0$}
    t \varphi'(t) \in
    L^1(0^+);
\end{equation}
\begin{equation}\label{Phi3}\tag{$\Phi3$}
    \exists \, B > 0,\,
    \theta \in (-\infty, 2) : \varphi'(ts) \geq
    B\varphi'(t)s^{-\theta}\,\,
    \forall \, t \in \erre^+, \forall \, s \in [1, +\infty).
\end{equation}
Integrating, it is easy to deduce that the following condition is
implied by \eqref{Phi3}:
\begin{equation}
    \varphi(ts) \geq
    B\varphi(t)s^{1-\theta}\,\,
    \forall \, t \in \erre^+, \forall \, s \in [1, +\infty),
\end{equation}
Note that $\varphi(t) = t^{p-1}, p
>1$ satisfies \eqref{Phi3} with $B=1$, $2-p\le \theta<2$. Again, by way of example,  if
\[
\varphi(t) = \int_0^t \frac{\mathrm{d}s}{P(s)},
\]
where $P(s)$ is a polynomial with degree at most $\theta$,
non-negative coefficients and such that $P'(0)>0$, then $\varphi$
satisfies \eqref{Phi3}. We would also like to stress that
conditions \eqref{Phi3} and \eqref{cond_ult_phi} are compatible,
as it is apparent,
for instance, for the $p$-Laplacian .\\
As in the previous theorems, the necessity of dealing with the
density function leads us to require a relaxed homogeneity also on
$g$, as expressed by the following inequality:
\begin{equation}\label{G}\tag{$G$}
  g(st) \leq \widetilde{D} s^{\tau +1} t^2 \varphi'(t) \quad \forall \, s\in [0,
  1], \, t \in \erre^+
\end{equation}
where $\tau$ is as in \eqref{cond_ult_phi} and $\tilde{D}$ is a
positive constant; this bound on $g$ is also due to a structural
constraint which comes from the construction of the supersolution.
Unfortunately, for the $p$-Laplacian this turns out to be quite
restrictive. For example, if $g(t)=Dt^\nu$, for some $0\le \nu$
and some constant $D>0$, it is not
hard to see that \eqref{G} holds if and only if $\nu=p$.
However, since \eqref{equaz_col_meno} is an inequality, solving for this $g$ will solve for any other smaller $g$.\\
We now examine the steps leading to the definition of the
Keller-Osserman condition adapted to inequality
\eqref{equaz_col_meno}. Setting $t=1$ in \eqref{Phi3} we have

\[
\varphi'(s) \geq B \varphi'(1) s^{-\theta},
\]
and since $\varphi'(1)>0$ we deduce, integrating and using
$\theta<2$,

\[
t \varphi'(t) \not \in L^1(+\infty).
\]
In the present case, $l \equiv 1$ and the definition of $K$ given
in \eqref{def_K} becomes

\[
K(t) = \int_0^t s \varphi'(s)\,\mathrm{d}s.
\]
It follows that \eqref{Phi3} with $\theta \leq 2$ implies that $K$
is a $C^1$-diffeomorphism from $\erre_0^+$ onto itself. From
\eqref{Phi3} we also have, for $s \in \erre^+, y \in [1,
+\infty)$,

\[
\int_0^t s y \varphi'(sy)\,\mathrm{d}s \geq B y^{1-\theta}
\int_0^t s \varphi'(s)\,\mathrm{d}s,
\]
so that
\begin{equation}\label{kappa}
  K(ty) \geq B y^{2-\theta}K(t) \qquad \forall \, t \in \erre^+,
  \, \forall \, y \in [1, +\infty).
\end{equation}
Next, we define
\begin{equation*}
  \hat{F}(t)=\int_0^tf(s)e^{(2-\theta)\int_0^sh(x)\,\mathrm{d}x}\,\mathrm{d}s.
\end{equation*}
For $s\in\erre^+$ we let
\begin{equation*}
  t=K^{-1}\pa{\sigma\hat{F}(s)}.
\end{equation*}
Since $K^{-1}$ is non-decreasing we get
\begin{equation*}
  y=\frac{K^{-1}\pa{\hat{F}(s)}}{K^{-1}\pa{\sigma\hat{F}(s)}}\geq1,
\end{equation*}
and applying inequality \eqref{kappa} we deduce
\begin{equation*}
  K\pa{K^{-1}\pa{\hat{F}(s)}}\geq
  BK\pa{K^{-1}\pa{\sigma\hat{F}(s)}}\sq{\frac{K^{-1}\pa{\hat{F}(s)}}{K^{-1}\pa{\sigma\hat{F}(s)}}}^{2-\theta}.
\end{equation*}
Hence we obtain
\begin{equation}
  \sq{\frac{K^{-1}\pa{\hat{F}(s)}}{K^{-1}\pa{\sigma\hat{F}(s)}}}^{2-\theta}\leq\frac{1}{B\sigma}.
\end{equation}
Since $\theta<2$ this can be written as
\begin{equation}
  \frac{\sigma^{\frac{1}{2-\theta}}}{K^{-1}\pa{\sigma\hat{F}(s)}}\leq
  \frac{B^{-\frac{1}{2-\theta}}}{K^{-1}\pa{\hat{F}(s)}},\quad
  s\in\erre^+.
\end{equation}
In conclusion, the following inequality holds:
\begin{equation}\label{ineq_4.5}
  \int^{+\infty}\frac{e^{\int_0^sh(x)\,\mathrm{d}x}}{K^{-1}\pa{\sigma\hat{F}(s)}}\,\mathrm{d}s\leq
  \pa{\frac{1}{B\sigma}}^{\frac{1}{2-\theta}}\int^{+\infty}
  \frac{e^{\int_0^sh(x)\,\mathrm{d}x}}{K^{-1}\pa{\hat{F}(s)}}\,\mathrm{d}s.
\end{equation}
We are now ready to introduce the further generalized
Keller-Osserman condition in the form

\begin{definition}
 The \textbf{\emph{generalized Keller-Osserman condition}} for
inequality
$$
\Delta^\varphi_{H^m}u \geq f(u) - h(u)g\pa{\absh{\gradh{u}}}
$$
is the request:
  \begin{equation}\label{generalized_hatK-O}\tag{$\hat{KO}$}
\frac{e^{\int_0^th(x)\,\mathrm{d}x}}{K^{-1}\pa{\hat{F}(t)}}\in L^1(+\infty).
  \end{equation}
\end{definition}
\vspace{0.4cm}
As we have already mentioned, the roles of $f$ and $h$ in the
above condition are far from being specular. In particular, $h$
has two opposite effects: on the one hand the explicit term
$e^{\int_0^t h(x)\,\mathrm{d}x}$ supports the non-integrability,
hence the existence, on the other hand its presence in the
expression for $\hat{F}(t)$ favours integrability.\par

We observe that, under assumptions \eqref{H} and \eqref{Phi3},
inequality \eqref{ineq_4.5} implies that, if
\eqref{generalized_hatK-O} holds, then for every $\sigma \in (0,
1]$

\begin{equation}\label{ineq_4.6}
 \frac{e^{\int_0^th(x)\,\mathrm{d}x}}{K^{-1}\pa{\sigma \hat{F}(t)}}\in L^1(+\infty).
\end{equation}
A particular case arises when $h \in L^1(+\infty)$. We are going
to see that, independently of the sign of $h$, condition
\eqref{generalized_hatK-O} and \ref{generalized_K-O} are indeed
equivalent:
\begin{proposition}\label{PR_4.7}
  Assume \eqref{Phi}, \eqref{F}, \eqref{Phi3} and suppose that $h: \erre_0^+\rightarrow \erre$ is a continuous
  function
  such that $h \in L^1(+\infty)$. Then
  \[
   \frac{e^{\int_0^th(x)\,\mathrm{d}x}}{K^{-1}\pa{\hat{F}(t)}}\in L^1(+\infty) \quad \text{if and only
  if}  \quad \frac{1}{K^{-1}\pa{F(t)}}\in L^1(+\infty).
  \]
  \begin{proof}
    First of all we observe that, since $\theta<2$,
    \[
    \hat{F}(t)=\int_0^tf(s)e^{(2-\theta)\int_0^sh(x)\,\mathrm{d}x}\,\mathrm{d}s
    \leq
    e^{(2-\theta)\|h\|_{L^1}}\int_0^tf(s)\,\mathrm{d}s=
    \Lambda_1 F(t)
    \]
    with $\Lambda_1 \geq 1$. Similarly $F(t) \leq \Lambda_2\hat{F}(t)$ with $\Lambda_2 \ge 1$.

    \noindent Thus, since  $K^{-1}$ is non-decreasing
    \begin{equation}\label{4.7}
      \int^{+\infty}\frac{\mathrm{d}s}{K^{-1}\pa{F(s)}}\leq
  \int^{+\infty}
  \frac{\mathrm{d}s}{K^{-1}\pa{\Lambda_1^{-1}\hat{F}(s)}}.
    \end{equation}
    We now perform the change of variables $t = s \Lambda_1^{-1}$.
    Thus
    \begin{equation}\label{4.8}
    \int^{+\infty}\frac{\mathrm{d}s}{K^{-1}\pa{\Lambda_1^{-1}\hat{F}(s)}}\leq
  \Lambda_1\int^{+\infty}
  \frac{\mathrm{d}t}{K^{-1}\pa{\Lambda_1^{-1}\hat{F}(\Lambda_1 t)}}.
    \end{equation}
    Since $\Lambda_1\ge 1$, denoting with $a(s)= f(s)e^{(2-\theta)\int_0^sh(x)\,\mathrm{d}x}$
    we have
    \[
    \hat{F}(\Lambda_1 t) = \int_0^{\Lambda_1 t}a(y)\,\mathrm{d}y =
    \Lambda_1\int_0^t a(\Lambda_1 x)\,\mathrm{d}x \geq \Lambda_1 e^{-(2-\theta)
    \|h\|_{L^1}}\int_0^t a(z)\,\mathrm{d}z =
    \Lambda\hat{F}(t)
    \]
    for some constant $0<\Lambda\le \Lambda_1$. Hence $\Lambda_1^{-1}\hat{F}(\Lambda_1 t)
    \ge \sigma \hat{F}( t)$, where $\sigma= \Lambda\Lambda_1^{-1}
    \le 1$. Using \eqref{4.7}, \eqref{4.8}, the monotonicity of $K^{-1}$ and Lemma
    \ref{sigma_lemma} (in particular inequality
    \eqref{sigmalemma2}) we show that
    \begin{equation}\label{4.9}
    \begin{array}{l}
     \displaystyle  \int^{+\infty}\frac{\mathrm{d}s}{K^{-1}\pa{F(s)}}\leq
  \int^{+\infty}
  \frac{\mathrm{d}s}{K^{-1}\pa{\Lambda_1^{-1}\hat{F}(s)}} \le \\[0.6cm]
  \displaystyle \le \Lambda_1\int^{+\infty}\frac{\mathrm{d}s}{K^{-1}(\sigma F(s))} \le \frac{\Lambda_1}{\sigma}
  \int^{+\infty}\frac{\mathrm{d}s}{K^{-1}(F(s))}.
 \end{array}
    \end{equation}
    Therefore, $h \in L^1(\erre^+)$ and \eqref{4.9} immediately
    imply that
    \[
    \frac{e^{\int_0^th(x)\,\mathrm{d}x}}{K^{-1}\pa{\hat{F}(t)}}\in L^1(+\infty) \qquad \text{if and only if}
    \qquad\frac{1}{K^{-1}\pa{F(t)}}\in
    L^1(+\infty).
    \]

  \end{proof}

\end{proposition}

\subsection{Construction of the supersolution and final steps}
Now we proceed with the construction of the supersolution; the
  idea follows  the lines of Lemma
  \ref{supersol}, but  we briefly reproduce the main steps.
\begin{lemma}\label{Le_4.11}
  Assume the validity of \eqref{Phi}, \eqref{F}, \eqref{H}, \eqref{Phi3} and of the Keller-Osserman assumption \eqref{generalized_hatK-O}. Fix
  $0 < \eps < \eta$, $0 < t_0 < t_1$. Then there exists $\sigma
  \in (0, 1]$, $T_\sigma > t_1$ and $\alpha : [t_0, T_\sigma) \ra
  [\eps, +\infty)$ satisfying
  \begin{equation}\label{4.12}
      \begin{cases}
        \pa{\varphi(\alpha')}' + \frac{2m+1}{t}\varphi(\alpha')
        \leq f(\alpha)-h(\alpha)(\alpha')^2\varphi'(\alpha'); \\
        \alpha'>0, \quad \alpha(t) \uparrow +\infty \quad\text{as}\quad t \ra T_\sigma^-,\\
        \alpha(t_0) = \eps\quad \text{and} \quad\alpha(t) \leq \eta \quad \text{on}\,\, [t_0,
        t_1].
      \end{cases}
    \end{equation}
    \begin{proof}
      First of all we observe that, using \eqref{generalized_hatK-O} and
      \eqref{ineq_4.6} we have that
      \[
      \int_\eps^{+\infty} \frac{e^{\int_0^sh(x)\,\mathrm{d}x}}{K^{-1}\pa{\sigma \hat{F}(s)}}\,\mathrm{d}s \, \uparrow +\infty \quad \text{as}
  \,\, \sigma \downarrow 0^+.
      \]
      We thus fix $\sigma_0 \in (0, 1]$ so that, for every $\sigma
      \in (0, \sigma_0]$
      \begin{equation}\label{4.13}
        T_\sigma = t_0 + \int_\eps^{+\infty} \frac{e^{\int_0^sh(x)\,\mathrm{d}x}}{K^{-1}\pa{\sigma \hat{F}(s)}}\,\mathrm{d}s > t_1.
      \end{equation}
     Implicitly define the $C^2$-function $\alpha : [t_0,
      T_\sigma) \ra [\eps, +\infty)$ by setting
      \begin{equation}\label{4.14}
        T_\sigma -t = \int_{\alpha(t)}^{+\infty} \frac{e^{\int_0^sh(x)\,\mathrm{d}x}}{K^{-1}\pa{\sigma
        \hat{F}(s)}}\,\mathrm{d}s.
      \end{equation}
      By construction, $\alpha(t_0) = \eps$ and $\alpha(t) \ra
      +\infty$ as $t \ra T_\sigma^-$. We differentiate
      \eqref{4.14} a first time to obtain
      \begin{equation}\label{4.14'}
        K^{-1}\pa{\sigma\hat{F}(\alpha)}=\alpha' e^{\int_0^\alpha h}
      \end{equation}
      so that $\alpha' >0$. Transforming the above into $\sigma\hat{F}(\alpha) = K\pa{\alpha' e^{\int_0^\alpha
      h}}$,
      differentiating once more and using the definition of
      $\hat{F}$ and $K$ we arrive at
      \begin{equation*}
      \sigma f(\alpha)e^{(2-\theta)\int_0^\alpha h}\alpha'=\alpha'e^{2\int_0^\alpha
      h}\varphi'\pa{\alpha'e^{\int_0^\alpha
      h}}\sq{\alpha''+\pa{\alpha'}^2h(\alpha)}.
      \end{equation*}
      We use \eqref{Phi3} and $\alpha'>0$ to deduce
      \begin{equation*}
        \sigma f(\alpha)\geq B\varphi'\pa{\alpha'}
        \sq{\alpha''+\pa{\alpha'}^2h(\alpha)}
      \end{equation*}
      and thus
      \begin{equation}\label{4.15}
        \varphi'\pa{\alpha'}\alpha''\leq\frac{\sigma}{B}f(\alpha)-\pa{\alpha'}^2\varphi'\pa{\alpha'}h(\alpha).
      \end{equation}
      Integrating \eqref{4.15} on $[t_0,t]$ and using $\alpha'>0$,
      $\varphi'\geq0$, \eqref{F} and \eqref{H} we obtain
      \begin{equation}\label{4.16}
        \varphi\pa{\alpha'(t)}\leq\varphi\pa{\alpha'(t_0)}+\frac{\sigma}{B}tf(\alpha(t)).
      \end{equation}
      Putting together \eqref{4.15} and \eqref{4.16} and using
      \eqref{F}
      \begin{align}\label{4.17}
        &\varphi'\pa{\alpha'}\alpha''+\frac{2m+1}{t}\varphi\pa{\alpha'}\leq\nonumber\\
         &\leq f(\alpha)
        \sq{\frac{\sigma}{B}2(m+1)+\frac{2m+1}{t_0}\frac{\varphi\pa{\alpha'(t_0)}}{f(\alpha(t_0))}}
        -\pa{\alpha'}^2h(\alpha)\varphi'\pa{\alpha'}.
      \end{align}
      From \eqref{4.14'}
      \begin{equation*}
        \alpha'(t_0)=K^{-1}\pa{\sigma\hat{F}(\eps)}e^{-\int_0^\eps
        h(x)\,\mathrm{d}x}.
      \end{equation*}
      Therefore, since $\varphi(t)\ra0$ as $t\ra0^+$, choosing
      $\sigma\in(0,\sigma_0]$ sufficiently small, \eqref{4.17}
      yields
      \begin{equation*}
        \varphi'\pa{\alpha'}\alpha''+\frac{2m+1}{t}\varphi\pa{\alpha'}\leq
        \frac{1}{D}f(\alpha)-h(\alpha)\pa{\alpha'}^2\varphi'\pa{\alpha'}
      \end{equation*}
      on $[t_0, T_\sigma)$. To prove that $\alpha(t)\leq\eta$ on
      $[t_0,t_1]$ we observe that
      \begin{equation*}
        t_1-t_0=T_\sigma-t_0+t_1-T_\sigma=\int_\eps^{\alpha(t_1)}
        \frac{e^{\int_0^sh(x)\,\mathrm{d}x}}{K^{-1}\pa{\sigma
        \hat{F}(s)}}\,\mathrm{d}s.
      \end{equation*}
      Hence, since the integrand goes monotonically to $+\infty$
      as $\sigma\ra0^+$, we need to have $\alpha(t_1)\ra\eps$ as
      $\sigma\ra0^+$. Since $\alpha'>0$ this proves the desired
      property.
    \end{proof}
\end{lemma}
We are now ready to state the non-existence result for inequality
\eqref{equaz_col_meno}. The proof is a minor modification of the
one given for Theorem \ref{non_existence_th}, therefore we only
sketch the main points referring to Section \ref{principale} for
definitions and notations.

\begin{theorem}\label{TH_non_exist_meno}
  Let $\varphi, f,  h, g$ satisfy  \eqref{Phi}, \eqref{F}, \eqref{H}, \eqref{G}, \eqref{Phi0},
  \eqref{cond_ult_phi}, \eqref{Phi3}, and \eqref{generalized_hatK-O}.
    Let $u$ be a non-negative $C^1$-solution of
    \begin{equation}
      \Delta^\varphi_{H^m}u \geq f(u)-h(u)g(\absh{\gradh{u}})\qquad
      \text{on}\,\,\,
      H^m.
    \end{equation}
    Then $u \equiv 0$.
\begin{proof}
First of all, note that it is sufficient to prove that $u$ is
equal to a constant $c$; indeed, by assumption \eqref{G}, $0 =
\Delta^\varphi_{H^m}c \geq f(c) - h(c)g(0) = f(c)$ and the
conclusion follows from \eqref{F}. Now we prove that a maximum
principle holds for equation \eqref{equaz_col_meno} on a domain
$\Omega$; indeed, if we assume $u(\tilde{q}) = u^*$ for some
$\tilde{q}\in \Omega$, then there exists a neighbourhood
$U_{\tilde{q}} \subseteq \Omega$ such that, for every $\eps>0$,
$g(\absh{\gradh{u}}) < \eps$ on $U_{\tilde{q}}$. This implies, up
to choosing $\eps$ sufficiently small, $\Delta^\varphi_{H^m}u \geq
f(u)-h(u^*)\eps \geq 0$ on $U_{\tilde{q}}$. Then, by Theorem
\ref{maximum_principle}, $u \equiv u^*$ on such neighbourhood, and
thus the set $\set{q \in \Omega : u(q) = u^*}$ is non-empty, open
and closed in $\Omega$; therefore, $u \equiv u^*$ in $\Omega$.

 Eventually, in order to prove the constancy
of $u$, assume, by contradiction, that there exists $q_0 \in H^m$
such that $u(q_0) < u^*$; then, by the maximum principle, $u <
u^*$ on $H^m$. We now proceed as in the proof of Theorem
\ref{non_existence_th} and define $r_0$, $\eta$, $\tilde{q}$,
$\tilde{r}$ in the same way. Then, we construct the function $v(q)
= \alpha(r(q))$, with $\alpha$ as in Lemma \ref{Le_4.11}. A
calculation shows that
\begin{align*}
\Delta^\varphi_{H^m}v
&=\sqrt{\psi}\sq{\sqrt{\psi}\varphi'\pa{\alpha'(r)\sqrt{\psi}}\alpha''(r)
+\frac{2m+1}{r}\varphi\pa{\alpha'(r)\sqrt{\psi}}}\le\\
&\le\pa{\sqrt{\psi}}^{1+\tau}D\sq{\varphi'\pa{\alpha'(r)}\alpha''(r)
+\frac{2m+1}{r}\varphi\pa{\alpha'(r)}}\le\\
&\le\pa{\sqrt{\psi}}^{1+\tau}D\sq{\frac{1}{D}f(\alpha)-h(\alpha)(\alpha')^2\varphi'(\alpha')}\le\\
&\le f(\alpha(r))-\frac{D}{\widetilde{D}}
h(\alpha(r))g(\alpha'(r)\sqrt{\psi})\leq
f(v)-h(v)g(\absh{\gradh{v}}),
\end{align*}
where in the last inequality we have used \eqref{G} and we have
chosen
$D$ in \eqref{cond_ult_phi} big enough to ensure $D \geq \bar{D}$. \\
If $\xi$ lies in the connected component $\Gamma_\mu$, using
\eqref{F}, \eqref{H} and $\absh{\gradh{u}(\xi)} =
\absh{\gradh{v}(\xi)}$ we obtain
\begin{align}
\Delta^\varphi_{H^m} u (\xi) & \ge
f(u(\xi))-h(u(\xi))g(\absh{\gradh{u}(\xi)}) > \\
& > f(v(\xi))-h(v(\xi))g(\absh{\gradh{v}(\xi)}) \ge
\Delta^\varphi_{H^m} v (\xi).
\end{align}
The rest of the proof is  much the same.
\end{proof}
\end{theorem}

\begin{remark}
\emph{We note that the maximum principle is indeed unnecessary for
the proof of the final steps in Theorems \ref{non_existence_th}
and \ref{TH_non_exist_meno}. If we assume that $u$ is not
constant, we can consider a point $q_0$ such that $u(q_0)< u^*$
and, by continuity, a small radius $r_o$ such that $u|_{\partial
B_{r_0}}(q_0) < u^*$. Using the invariance property, we can
consider $q_0$ as the origin for the Koranyi distance, and proceed
analogously to the end.}

\end{remark}

As for Theorem \ref{non_existence_th}, we can state the Euclidean
counterpart of Theorem \ref{TH_non_exist_meno} substituting
assumption \eqref{G}  with the request
\begin{equation}\label{geuclidea}
    g(t) \le Dt^2\varphi'(t) \quad \text{on } (0,+\infty).
    \tag{$\widetilde{G}$}
\end{equation}

\noindent We have:
\begin{theorem}\label{TH_non_exist_meno euclideo}
  Let $\varphi, f,  h, g$
  satisfy \eqref{Phi}, \eqref{F}, \eqref{H}, \eqref{geuclidea}, \eqref{Phi0}, \eqref{cond_ult_phi},
  \eqref{Phi3},  and \eqref{generalized_hatK-O}.
    Let $u\in C^1(\erre^m)$ be a non-negative solution of
    \begin{equation}
      \Delta^\varphi_{\erre^m}u \geq f(u)-h(u)g(|\nabla{u}|)\qquad
      \text{on}\,\,\,
      \erre^m.
    \end{equation}
    Then $u\equiv 0$.
\end{theorem}

\subsection{Another existence result for the $p$-Laplacian}
As a quick application of Lemma \ref{PR_4.7} and Theorem
\ref{th_existence}, we can deduce that the modified
Keller-Osserman condition \eqref{generalized_hatK-O} is optimal in
the case of the $p$-Laplacian.
\begin{theorem}\label{esistenzameno}
 Let $f,  h, g$ satisfy \eqref{F}, \eqref{H}, \eqref{G},
  \eqref{cond_ult_phi} and \eqref{Phi3} with $\tau=0$. Furthermore suppose
  that $h \in L^1(\erre^+)$. Then, the following conditions are
equivalent:
\begin{itemize}
    \item [i)] there exists a non-negative, non-constant
        solution $u \in C^1(H^m)$ of inequality
        $\Delta^p_{H^m} u \ge f(u) - h(u)g(\absh{\gradh{u}});$
    \item [ii)] $\displaystyle \frac{1}{K^{-1}(F(t))} \not\in
        L^1(+\infty)$.
 \end{itemize}
\end{theorem}
\vspace{0.2cm}

\begin{proof}

First, we deduce from the assumptions and from Lemma \ref{PR_4.7}
the equivalence between \eqref{generalized_K-O} and
\eqref{generalized_hatK-O}. We have already pointed out that the
$p$-Laplacian satisfies \eqref{cond_ult_phi} for every $0\le \tau
\le p-1$: as it can be checked, the choice of $\tau=0$ is the
least stringent on \eqref{G}. Furthermore, \eqref{Phi0} is
authomatic. This shows that implication $i) \Rightarrow ii)$ is an
immediate application of Theorem \ref{TH_non_exist_meno}.
Regarding the other one, set $l(t) \equiv 1$ and apply the
existence part of Theorem \ref{th_existence} (note that all the
assumptions are satisfied), to get a solution of
$$
\Delta^p_{H^m}u \ge f(u).
$$
Since the RHS is trivially greater than
$f(u)-h(u)g(\absh{\gradh{u}})$ we have the desired conclusion.

\end{proof}

\bibliographystyle{amsplain}
\bibliography{the_bibliography}

\providecommand{\bysame}{\leavevmode\hbox to3em{\hrulefill}\thinspace}
\providecommand{\MR}{\relax\ifhmode\unskip\space\fi MR }
\providecommand{\MRhref}[2]{%
  \href{http://www.ams.org/mathscinet-getitem?mr=#1}{#2}
}
\providecommand{\href}[2]{#2}
\begin{thebibliography}{10}

\bibitem{BPT}
M.~Biroli, C.~Picard, and N.~A. Tchou, \emph{Homogenization of the
  $p$-{L}aplacian associated with the {H}eisenberg group}, Rend. Accad. Naz.
  Sci. XL Mem. Mat. Appl. (5) \textbf{\textbf{22}} (1998), 23--42.

\bibitem{BON}
J.~M. Bony, \emph{Principe du maximum, inégalité de {H}arnack et unicité du
  problème de cauchy pour les opérateurs elliptiques dégénérés}, Ann. Inst.
  Fourier (Grenoble) \textbf{\textbf{19}} (1969), no.~1, 277--304.

\bibitem{BRS}
L.~Brandolini, M.~Rigoli, and A.~G. Setti, \emph{Positive solutions of
  {Y}amabe-type equations on the {H}eisenberg group}, Duke Math. J.
  \textbf{\textbf{91}} (1998), no.~2, 241--295.

\bibitem{CDG}
L.~Capogna, D.~Danielli, and N.~Garofalo, \emph{Capacitary estimates and the
  local behavior of solutions of nonlinear subelliptic equations}, Amer. J.
  Math. \textbf{\textbf{118}} (1997), 1153--1196.

\bibitem{DOM}
A.~Domokos, \emph{Differentiability of solutions for the non-degenerate
  $p$-{L}aplacian in the {H}eisenberg group}, J. Differential Equations
  \textbf{\textbf{204}} (2004), 439--470.

\bibitem{GALA}
N.~Garofalo and E.~Lanconelli, \emph{Frequency functions on the {H}eisenberg
  group, the uncertainty principle and unique continuation}, Ann. Inst. Fourier
  (Grenoble) \textbf{\textbf{40}} (1990), no.~2, 313--356.

\bibitem{HOR}
L.~Hörmander, \emph{Hypoelliptic second order differential equations}, Acta
  Math. \textbf{\textbf{119}} (1967), 147--171.

\bibitem{OSS}
R.~Osserman, \emph{On the inequality $\triangle u\geq f(u)$}, Pacific J. Math.
  \textbf{\textbf{7}} (1957), 1641--1647.

\bibitem{PRS_MEM}
S.~Pigola, M.~Rigoli, and A.~G. Setti, \emph{Maximum principles on {R}iemannian
  manifolds and applications}, Mem. Amer. Math. Soc., vol. \textbf{822}, Amer.
  Math. Soc., Providence, RI, 2005.

\bibitem{PRS}
P.~Pucci, M.~Rigoli, and J.~Serrin, \emph{Qualitative properties for solutions
  of singular elliptic inequalities on complete manifolds}, J. Differential
  Equations \textbf{\textbf{234}} (2007), no.~2, 507--543.

\bibitem{PUSE}
P.~Pucci and J.~Serrin, \emph{The strong maximum principle revisited}, J.
  Differential Equations \textbf{\textbf{196}} (2004), 1--66.

\end{thebibliography}

\end{document}